\documentclass[a4paper,10pt]{amsart}
\usepackage{amsthm,amsmath,amssymb,amsfonts}
\usepackage[T1]{fontenc}
\usepackage[utf8]{inputenc}
\usepackage{verbatim}
\usepackage{marvosym}
\usepackage{stackrel}
\usepackage{graphicx}
\usepackage{eucal}
\usepackage{bbold}
\usepackage[svgnames]{xcolor}
\definecolor{blue(munsell)}{rgb}{0.0, 0.5, 0.69}
\usepackage{lipsum}
\usepackage{bbm}
\usepackage{caption}
\usepackage{enumitem}
\usepackage{mathtools}
\usepackage{epigraph}
\usepackage{color}
\usepackage{microtype}
\usepackage{relsize}
\usepackage{amsfonts}
\usepackage{adjustbox}
\usepackage{subfig}
\usepackage{framed}
\usepackage{hyperref}
\hypersetup{
    colorlinks = true,
    linkbordercolor = {pink},
   	linkcolor =[HTML]{4e6692},
	anchorcolor = {pink},
	citecolor =  [HTML]{a57e0a},
	filecolor = [HTML]{4e6692},
	menucolor = [HTML]{4e6692},
	runcolor =  [HTML]{4e6692},
	urlcolor = [HTML]{4e6692},
}
\usepackage[capitalise]{cleveref}
\usepackage{trimclip}

\DeclareFontFamily{U}{min}{}
\DeclareFontShape{U}{min}{m}{n}{<-> udmj30}{}
\newcommand{\yo}{\!\text{\usefont{U}{min}{m}{n}\symbol{'210}}\!}

\usepackage{stmaryrd}
\usepackage{tikz}
\usepackage{tikz-cd}
\usepackage{quiver}

\theoremstyle{definition}
\newtheorem{thm}{Theorem}[subsection]
\newtheorem*{thm*}{Theorem}
\newtheorem{prop}[thm]{Proposition}
\newtheorem*{prop*}{Proposition}
\newtheorem{lem}[thm]{Lemma}

\newtheorem{defn}[thm]{Definition}
\newtheorem{war}[thm]{Warning}
\newtheorem*{war*}{Warning}
\newtheorem{rem}[thm]{Remark}

\newtheorem{exa}[thm]{Example}

\newcommand{\textdel}[1]{}

\newcommand{\Hcal}{\mathcal{H}}

\newcommand{\Set}{\mathsf{Set}}
\newcommand{\Fin}{\mathsf{Fin}}
\newcommand{\op}{{}^{\mathrm{op}}}
\DeclareMathOperator{\id}{id}

\newcommand{\mc}[1]{\mathcal{#1}}
\newcommand{\mb}[1]{\mathbf{#1}}
\newcommand{\mbb}[1]{\mathbb{#1}}
\newcommand{\mr}[1]{\mathrm{#1}}

\newcommand{\ms}[1]{\mathsf{#1}}
\newcommand{\lan}{\ms{lan}}
\newcommand{\ran}{\ms{ran}}
\newcommand{\nt}{\Rightarrow}

\newcommand{\WRInj}{\ms{WRInj}}
\newcommand{\Topoi}{\ms{Topoi}}

\newcommand{\Lex}{\ms{Lex}}
\newcommand{\lex}{\ms{lex}}
\newcommand{\Cat}{\ms{Cat}}

\newcommand{\psh}{\ms{Psh}}

\newcommand{\alg}{\ms{alg}}
\newcommand{\surj}{\twoheadrightarrow}
\newcommand{\inj}{\rightarrowtail}
\newcommand{\hook}{\hookrightarrow}

\newcommand{\ov}[1]{{\overline{#1}}}
\newcommand{\qsi}[1]{\widetilde{#1}}

\newcommand{\pair}[1]{\left\langle#1\right\rangle}
\newcommand{\inv}{^{-1}}

\newcommand{\dv}{\uparrow}

\newcommand{\set}[1]{\{#1\}}
\newcommand{\Cl}{\ms{Cl}}

\newcommand{\sub}{\mr{Sub}}
\newcommand{\N}{\mbb N}
\newcommand{\Frm}{\mb{Frm}}
\newcommand{\DL}{\mb{DL}}
\newcommand{\rg}[1]{\mc O_{#1}}
\newcommand{\elem}{\int\!\!}
\newcommand{\syn}{\ms{Syn}}

\makeatletter
\newcommand{\ct@}[2]{%
  \vtop{\m@th\ialign{##\cr
    \hfil$#1\operator@font lim$\hfil\cr
    \noalign{\nointerlineskip\kern1.5\ex@}#2\cr
    \noalign{\nointerlineskip\kern-\ex@}\cr}}%
}
\newcommand{\ct}{%
  \mathop{\mathpalette\ct@{\rightarrowfill@\textstyle}}\nmlimits@
}
\makeatother
\makeatletter
\newcommand{\lt@}[2]{%
  \vtop{\m@th\ialign{##\cr
    \hfil$#1\operator@font lim$\hfil\cr
    \noalign{\nointerlineskip\kern1.5\ex@}#2\cr
    \noalign{\nointerlineskip\kern-\ex@}\cr}}%
}
\newcommand{\lt}{%
  \mathop{\mathpalette\lt@{\leftarrowfill@\textstyle}}\nmlimits@
}
\makeatother

\title{Craig Interpolation for subgeometric logics}

\author{Ivan Di Liberti$^{\mathbin{\rotatebox[origin=c]{-45}{$\diamond$}}}$}
\thanks{
$^{{\mathbin{\rotatebox[origin=c]{-45}{$\diamond$}}}}$ Department of Philosophy, Linguistics and Theory of Science, University of Gothenburg, Gothenburg, Sweden. \hfill \href{mailto:diliberti.math@gmail.com}{\sf diliberti.math@gmail.com}
}

\author{Lingyuan Ye$^\diamond$}
\thanks{
$^\diamond$ Department of Computer Science and Technology, University of Cambridge, Cambridge, UK. (Corresponding author). \hfill  \href{mailto:ye.lingyuan.ac@gmail.com}{\sf ye.lingyuan.ac@gmail.com}
}

\begin{document}

\begin{abstract}
We show that a vast class of finitary fragments of geometric logic admit a form of Craig interpolation property. In doing so, we provide a new dictionary to import technology from algebraic logic to categorical logic.
 
\end{abstract}

\maketitle

   {
   \hypersetup{linkcolor=black}
   \tableofcontents
   }

\section*{Introduction}

In 1957 Craig presented his interpolation theorem \cite{craig1957three,craig1957linear}. Roughly speaking, he proved that given a provable propositional sequent $\varphi(\ov p,\ov r) \vdash \psi(\ov p,\ov s)$, where the appearing formulas share a common language, in the propositional case a list of variables $\ov p$, there exists an \textit{interpolant} formula $\chi(\ov p)$,  mentioning only the common language, such that 
\[ \varphi(\ov p,\ov r) \vdash \chi(\ov p), \quad \chi(\ov p) \vdash \psi(\ov p,\ov s). \]
This result proved to be extremely influential in several areas of logic, including those of more computational value like \textit{model checking} \cite{mcmillan2005applications}, and those of more theoretical ambition. Indeed, Craig interpolation for propositional first order logic implies Beth definability theorem \cite{gabbay2010interpolation}.

Because of its versatility and its applications, Craig-interpolation-type of results became a whole research theme. Only two years after Craig's seminal contribution, Lyndon improved the theorem with applications to properties preserved under homomorphisms \cite{lyndon1959interpolation,lyndon1959properties}. In 1962 Shütte \cite{schutte1962interpolationssatz} extended the theorem to \textit{intuitionistic} logic with a proof theoretic technique, importing some previous ideas due of Maehara. Since the 1960s, this research topic of research has branched in many different directions, far too many to be surveyed here, we shall refer the interested reader to the forthcoming book \cite{ten2026theory} focusing on this very topic (see also \cite{vaananen2008craig}).

Among the most prominent variations of Craig's theorem, we find the work of Pitts \cite{pitts1983amalgamation,Pitts1983AnAO, pitts2020interpolation}, who delivered an algebraically flavoured version of Craig interpolation for propositional and first-order \textit{intuitionistic} logic, and for the coherent fragment of first-order logic as well, from a categorical logic perspective. The definition of interpolation in categorical logic terms is a generalisation of the original interpolation in its syntactic form as displayed earlier, and also of the one used in algebraic logic (cf.~\cite{hoogland2001definability}) which primarily applies to propositional fragments. Pitts' research line imports techniques coming from topos theory, especially various powerful descent-type theorems. In the attempt of framing Pitt's theorem, other perspectives have emerged on the topic since then, including \cite{galli2003strong,makkai1995gabbay,vcubric1993results,pavlovi1991categorical}. 

\textbf{The aim of this paper} is to explore a cluster of fragments of \textit{geometric logic} and assess to what extent these fragments admit a form of Craig interpolation theorem. Geometric logic is a prominent area of (categorical) logic, and includes the coherent fragment, for which Pitts proved Craig interpolation. 

There are several different motivations for this work. On the side of applications, it generalises Pitts' original contributions to other fragments of logic, now including -- for example-- regular logic. On a more speculative side, it expands our understanding of the geography of fragments of geometric logic. We will discuss later that this investigation has led to new ideas. Finally, we see this work as the first one in bridging the wisdom of algebraic logic to the technologies of categorical logic, expanding many classical results holding in propositional logic to the predicate case.

The main difference between both the aim and technique of this paper and other interpolation-type theorem is that, instead of trying to prove interpolation for \emph{specific} fragment of logic, we establish \emph{uniformly} a Craig-interpolation-type theorem for a wide class of subfragments of geometric logic. In particular, the techniques used in aforementioned Pitts' proof, though powerful, are highly tailored towards the coherent and intuitionistic logic, and cannot be directly applied to other fragments. 

It goes without saying that in order to accomplish our goal, we need a working notion of \textit{fragment} of geometric logic. The existence of such framework, where one can study logics and its fragments, is part of a broad research programme whose aim is to study \textit{logics} and their interaction started by the first author in recent times. A germinal version of this idea goes back to Power \cite{power1995tricategories}, and is starting to surface in recent years \cite{di2025bi, coraglia2021context, di2024sketches, osser2025a2sketchy}. 
For our purposes, though, such framing is too ambitious, and we can be content with carving fragments \textit{inside} geometric logic. In \cite{di2025logic}, the authors introduce the notion of \textit{fragment} of geometric logic which is perfectly suitable for our purposes. Our main theorem will be the following.

\begin{thm*}[\ref{interpolationofsubfragment}]
Let $\mathcal{H}$ be a fragment of  geometric logic between the regular and coherent fragment having an \'etale classifier. Then $\mc H$ has the interpolation property.
\end{thm*}

We will see in the following discussion that this is far from being the only interesting result of this paper, and that several auxiliary notions we shall develop may turn out to be even more interesting than the main result of the paper. 

\subsection*{Contributions and structure of the paper}
In \Cref{sec:prom}, following the generalisation of interpolation presented in~\cite{pitts1983amalgamation,pitts2020interpolation} from the perspective of categorical logic, we introduce the notion of interpolation we are considering for general \emph{doctrines} (a.k.a lax-idempotent pseudo-monads) on the 2-category of left exact categories $\Lex$. We also explain its relationship with the more syntactic notion of interpolation.

In our investigation, we have identified a property of logic that plays a key role in establishing interpolation results. In the language of doctrine, the property states that it should \emph{preserve slicing}. This is indeed such a fundamental property of doctrines associated with logic and type theory, that perhaps has not been paid enough attention to in the literature. In \Cref{sec:slicing}, we initiated a first step of studying the operation of slicing for doctrines over lex categories. One point worth mentioning is that, the slicing operation in this context, unlike that in $\Cat$, is a \emph{colimit} rather than a \emph{limit}. In particular, taking the slice of a syntactic category of a theory corresponds syntactically to adding a constant, thus it has a mapping-out universal property.

As a first application, in \Cref{sec:slicing} we prove the following propositional-bootstrap result, that relates the interpolation property of first-order doctrines with interpolation for certain corresponding lattice structures:

\begin{thm*}[\ref{reducetosub1}]
    Let $\ms T$ be a doctrine on lex categories which preserves slicing. Then it has the interpolation property iff for any cocomma square in $\alg(\ms T)$, the image under the 2-functor $\sub_{-}(1)$ has interpolation in the sense of~\cref{interpos}.
\end{thm*}

\Cref{sec:finitary} proceeds to provide a classification of the interpolation property for doctrines preserving slicing as an \emph{exactness property}. This aligns with the philosophy of algebraic logic, which studies the interpolation property as properties of algebras corresponding to a fragment of propositional logic. In particular, we will introduced a notion of \emph{t-conservative maps} of lex categories (\cref{tcons}), and show that it belongs to an orthogonal factorisation system for any \emph{finitary} doctrine on $\Lex$. Using this, we are able to provide a classification as below:

\begin{thm*}[\ref{tstableinterpolation}]
   Let $\ms T$ be a finitary doctrine on lex categories preserving slicing. It has the interpolation property iff t-conservative maps are closed under cocomma in $\alg(\ms T)$.
\end{thm*}

The next two sections will shift the focus from doctrines to fragments of geometric logic in the sense of \cite{di2025logic}. In \Cref{sec:etale}, after recalling the general theory of semantics prescription, we introduce the notion of classifier of \'etale maps (\cref{absoluteetale}), and prove the following theorem, connecting logics with a classifier to doctrines preserving slicing:

\begin{prop*}[\ref{etaleimpliessclicing}]
    If a fragment $\mc H$ has an \'etale map classifier, then $\ms T^{\mc H}$ preserves slicing.
\end{prop*}

Doctrines with a classifier abound in the literature, encompassing the most relevant examples. That being said, we cannot tell precisely how common this feature is, for we currently lack a counterexample, which could hide in some exotic behaviour that has never been observed before. Such investigation would be quite interesting.

In \Cref{sec:intfin} we deliver the main theorem of the paper, and we proceed to comment it in the remarks that follow it.

\begin{thm*}[\ref{interpolationofsubfragment}]
Let $\mathcal{H}$ be a fragment of  geometric logic between the regular and coherent fragment having an \'etale classifier. Then $\mc H$ has the interpolation property.
\end{thm*}

We finish the paper with a short \Cref{newintforcoh}, offering a new proof that the doctrine associated to coherent logic has interpolation (\Cref{interpolationpretopoi}). As we have discussed, this result was originally due to Pitts, who provided a proof using a version of Makkai's topos of filters. Our proof is much simpler and only requires the notion of classifying topos.

\subsection*{Comments and Further directions}
As we have discussed, besides the technical achievement of proving Craig interpolation, we understand this paper as a grounding to import ideas coming from algebraic logic to categorical logic (from posets to categories) with the intention of generalising several results in propositional logic to the predicate case. The authors had presented this program already in \cite[7.4]{di2025logic}. To some extent, this paper is a success in this direction. We will observe this in \Cref{hoog}. Yet, from other perspectives, our work remains unsatisfactory. 

Indeed our notion of fragment of geometric logic remains \textit{non-syntactic}, as was already discussed in \cite[7.3]{di2025logic} and it is thus hard to predict to how many logics our main results applied. It goes without saying that besides the examples that we present explicitly, we cannot tell precisely how logics between regular and coherent even look like. We discuss this limitation of our work in  \Cref{lim1}.

Finally, we also discuss in \Cref{lim2} that despite our analysis being inherently modular, and thus ready to be generalised to other contexts -- at least in principle -- it is rather unclear that our approach can cover every possible Craig-interpolation-style of result.

\begin{war*}
    Throughout the paper, when we work within a $2$-category, the limit and colimit notion we consider are always \emph{pseudo} in nature, i.e.\ we do not consider the corresponding \emph{strict} notions. For instance, by pullback in a 2-category, we always mean a \emph{pseudo}-pullback, rather than its strict companion. 
\end{war*}

\subsection*{Acknowledgments}
 The first named author received funding from Knut and Alice Wallenberg Foundation, grant no. 2020.0199. Both authors are grateful to Nathanael Arkor and Axel Osmond for their comments on a first draft of this paper.

\section{Interpolation from algebraic to categorical logic} \label{sec:prom}

We shall start by discussing the interpolation property for logics from a categorical perspective. First we will reformulate of the interpolation property in the propositional case (\Cref{interpos}) importing some wisdom of algebraic logic. Then we shall generalise such notion to the predicate case (\Cref{interpcat}). This is made possible by the syntactic category construction associated to theories, or equivalently the Lindenbaum-Tarski algebra for propositional theories. At the end of the section we will introduce a convenient notion of \textit{logic}, which we shall refer to as \textit{doctrine} (\Cref{doctrinesdef}) and define an appropriate notion of doctrine satisfying the interpolation property (\Cref{intfordoc}). A characterisation of these gadget will be our main effort throughout the entire paper.

\subsection{Categorical notions of Interpolation}
Abstractly, the interpolation property for propositional logics can be phrased in terms of the following interpolation property for \emph{posets}.

\begin{defn}[Interpolation for posets]\label{interpos}
    Consider a lax square between posets, i.e.\ a square below where $vg \le uf$ in the pointwise order,
    \[\begin{tikzcd}
    	C & A \\
    	B & D
    	\arrow["f", from=1-1, to=1-2]
    	\arrow["g"', from=1-1, to=2-1]
    	\arrow["\le"{description}, draw=none, from=1-1, to=2-2]
    	\arrow["u", from=1-2, to=2-2]
    	\arrow["v"', from=2-1, to=2-2]
    \end{tikzcd}\]
    We say it has \emph{interpolation} if for any $b \in B$ and $a\in A$, if $vb \le ua$, then there exists $c\in C$ that $b \le gc$ and $fc \le a$.
\end{defn}

\begin{rem}[Syntactic reformulation]\label{syntacrefprop}
    For the usual classical or intuitionistic propositional logic, the interpolation property is equivalent to certain squares between Boolean algebras or Heyting algebras having the interpolation property in the sense above. For instance, let $\varphi(\ov p,\ov r)$ and $\psi(\ov p,\ov s)$ be two classical formulas. The usual interpolation property for classical logic states that if the following sequent is provable,
    \[ \varphi(\ov p,\ov r) \vdash \psi(\ov p,\ov s), \]
    then there exists a formula $\chi(\ov p)$, such that 
    \[ \varphi(\ov p,\ov r) \vdash \chi(\ov p), \quad \chi(\ov p) \vdash \psi(\ov p,\ov s). \]
    Let $B(\ov p)$ be the free Boolean algebra generated by a lists of objects $\ov p$. The above interpolation property equivalently states that the following \emph{pushout square} of Boolean algebras has the interpolation property in the sense of~\cref{interpos}.
    \[\begin{tikzcd}
    	{B(\ov p)} & {B(\ov p,\ov s)} \\
    	{B(\ov p,\ov r)} & {B(\ov p,\ov r,\ov s)}
    	\arrow[from=1-1, to=1-2]
    	\arrow[from=1-1, to=2-1]
    	\arrow[from=1-2, to=2-2]
    	\arrow[from=2-1, to=2-2]
        \arrow["\lrcorner"{anchor=center, pos=0.125, rotate=180}, draw=none, from=2-2, to=1-1]
    \end{tikzcd}\]
\end{rem}

\begin{rem}[Lax squares v.s. commutative squares]\label{laxsquare}
    Notice that for squares of morphisms between Boolean algebras or Heyting algebras, the more general treatment by looking at \emph{lax} squares doesn't do much, because the negation operator is \emph{contravariant}, thus if $f,g : B \to B'$ are two morphisms between Boolean algebras and $f \le g$ pointwise, then $f = g$. However, this generalisation will turn out to be crucial when we consider \emph{positive fragments}, --say-- propositional logics only involving $\top,\wedge,\bot,\vee$, or correspondingly interpolation squares for \emph{distributive lattices}. This paper concerns with subfragments of \emph{geometric logic}~\cite[D1]{elephant2},~\cite[Ch. 1]{caramello2018theories}, where formulas are built out of \emph{positive} logical operators. Hence the same comment applies to the definition of interpolation of \emph{categories} in~\cref{interpcat} below as well.
\end{rem}

For first-order theories, we need to formulate an interpolation property for (lax) squares of \emph{categories}, which is slightly more involved:

\begin{defn}[Interpolation for left exact categories]\label{interpcat}
    Let $\Lex$ be the 2-category of left exact categories. Consider a lax diagram in $\Lex$ below,
    % https://q.uiver.app/#q=WzAsNCxbMCwwLCJcXG1jIEMiXSxbMSwwLCJcXG1jIEEiXSxbMCwxLCJcXG1jIEIiXSxbMSwxLCJcXG1jIEQiXSxbMCwxLCJmIl0sWzAsMiwiZyIsMl0sWzEsMywidSJdLFsyLDMsInYiLDJdLFs3LDYsIlxcYWxwaGEiLDAseyJjdXJ2ZSI6LTEsInNob3J0ZW4iOnsic291cmNlIjoyMCwidGFyZ2V0IjoyMH19XV0=
    \[\begin{tikzcd}
    	{\mc C} & {\mc A} \\
    	{\mc B} & {\mc D}
    	\arrow["f", from=1-1, to=1-2]
    	\arrow["g"', from=1-1, to=2-1]
    	\arrow[""{name=0, anchor=center, inner sep=0}, "u", from=1-2, to=2-2]
    	\arrow[""{name=1, anchor=center, inner sep=0}, "v"', from=2-1, to=2-2]
    	\arrow["\alpha", curve={height=-6pt}, between={0.2}{0.8}, Rightarrow, from=1, to=0]
    \end{tikzcd}\]
    We say this lax square has \emph{interpolation}, if for any $X\in\mc C$, the lax square on posets below has interpolation in the sense of~\cref{interpos},
    \[
    \begin{tikzcd}
        \sub_{\mc C}(X) \ar[d, "g"'] \ar[r, "f"] & \sub_{\mc A}(fX) \ar[r, "u"] & \sub_{\mc D}(ufX) \ar[d, "\alpha^*"] \\
        \sub_{\mc B}(gX) \ar[rr, "v"'] & & \sub_{\mc D}(vgX)
        \arrow["\le"{description}, draw=none, from=1-1, to=2-3]
    \end{tikzcd}
    \]
    Here $\alpha : vgX \to ufX$ is the component of the 2-cell $\alpha$ at $X$, and the map 
    \[ \alpha^* : \sub_{\mc D}(ufX) \to \sub_{\mc D}(vgX) \]
    is the induced map on subobjects by pulling back along $\alpha$. The above square is a lax square because for any subobject $U \hook X$ in $\mc C$, naturality of $\alpha$ gives us the following commuting square,
   % https://q.uiver.app/#q=WzAsNCxbMSwxLCJ1ZlgiXSxbMCwxLCJ2Z1giXSxbMCwwLCJ2Z1UiXSxbMSwwLCJ1ZlUiXSxbMSwwLCJcXGFscGhhIiwyXSxbMiwxLCIiLDIseyJzdHlsZSI6eyJ0YWlsIjp7Im5hbWUiOiJob29rIiwic2lkZSI6InRvcCJ9fX1dLFszLDAsIiIsMCx7InN0eWxlIjp7InRhaWwiOnsibmFtZSI6Imhvb2siLCJzaWRlIjoidG9wIn19fV0sWzIsMywiXFxhbHBoYSJdXQ==
    \[\begin{tikzcd}
    	vgU & ufU \\
    	vgX & ufX
    	\arrow["\alpha", from=1-1, to=1-2]
    	\arrow[hook, from=1-1, to=2-1]
    	\arrow[hook, from=1-2, to=2-2]
    	\arrow["\alpha"', from=2-1, to=2-2]
    \end{tikzcd}\]
    which implies $vgU\le\alpha^*ufU$ as subobjects of $vgX$.
\end{defn}

\begin{rem}[Syntactic reformulation]\label{syntacreffis}
    For concreteness, let us assume again we work with classical first-order logic for the moment. The usual formulation of interpolation in logic is the following. Let $\Sigma_0,\Sigma_1$ be two signatures. Concretely, a signature contains a set of \emph{sorts, relation symbols,} and \emph{function symbols}. Let $\varphi_i(\ov x)$ in the language of $\Sigma_i$ for $i = \set{0,1}$, where the \emph{free variables are from the common sorts} of $\Sigma_0$ and $\Sigma_1$.\footnote{This condition is usually not explicit in traditional discussion on interpolation of first-order logic, because usually only single sorted signatures are considered.} If the sequent $\varphi_0(\ov x) \vdash \varphi_1(\ov x)$ is provable, then there is a formula $\varphi(\ov x)$, again over the same context of variables, but belong to the common signature $\Sigma_0\cap\Sigma_1$, where the two sequents $\varphi_0(\ov x) \vdash \varphi(\ov x)$ and $\varphi(\ov x) \vdash \varphi_1(\ov x)$ are both provable.

    To see how the above interpolation property manifests itself in our categorical formulation, consider the following pushout diagram, where $\mc B[\Sigma]$ denotes the Boolean syntactic category generated by the signature $\Sigma$,
    % https://q.uiver.app/#q=WzAsNCxbMCwwLCJcXG1jIENbXFxTaWdtYV8wXFxjYXBcXFNpZ21hXzFdIl0sWzAsMSwiXFxtYyBDW1xcU2lnbWFfMF0iXSxbMSwwLCJcXG1jIENbXFxTaWdtYV8xXSJdLFsxLDEsIlxcbWMgQ1tcXFNpZ21hXzBcXGxlXFxTaWdtYV8xXSJdLFswLDIsImlfMSIsMCx7InN0eWxlIjp7InRhaWwiOnsibmFtZSI6Im1vbm8ifX19XSxbMCwxLCJpXzAiLDIseyJzdHlsZSI6eyJ0YWlsIjp7Im5hbWUiOiJtb25vIn19fV0sWzIsMywial8xIiwwLHsic3R5bGUiOnsidGFpbCI6eyJuYW1lIjoibW9ubyJ9fX1dLFsxLDMsImpfMCIsMix7InN0eWxlIjp7InRhaWwiOnsibmFtZSI6Im1vbm8ifX19XSxbNyw2LCIiLDEseyJjdXJ2ZSI6LTEsInNob3J0ZW4iOnsic291cmNlIjoyMCwidGFyZ2V0IjoyMH19XV0=
    \[\begin{tikzcd}
    	{\mc B[\Sigma_0\cap\Sigma_1]} & {\mc B[\Sigma_1]} \\
    	{\mc B[\Sigma_0]} & {\mc B[\Sigma_0\cup\Sigma_1]}
    	\arrow["{i_1}", tail, from=1-1, to=1-2]
    	\arrow["{i_0}"', tail, from=1-1, to=2-1]
    	\arrow["{j_1}", tail, from=1-2, to=2-2]
    	\arrow["{j_0}"', tail, from=2-1, to=2-2]
        \arrow["\lrcorner"{anchor=center, pos=0.125, rotate=180}, draw=none, from=2-2, to=1-1]
    \end{tikzcd}\]
    The fact that $\varphi_0(\ov x)$ and $\varphi_1(\ov x)$ share the same context of variables means that there is an object $X = X_1 \times \cdots X_n$ in $\mc B[\Sigma_0\cap\Sigma_1]$, viz. the interpretation for the sorts of the list of variables $\ov x$, such that $\varphi_0(\ov x)$ is a subobject of $i_0X$ in $\mc C[\Sigma_0]$, and similarly $\varphi_1(\ov x)$ is a subobject of $i_1X$. The fact that $\varphi_0(\ov x) \vdash \varphi_1(\ov x)$ is provable simply means that $j_0\varphi(\ov x) \le j_1\varphi_1(\ov x)$ as a subobject of $j_0i_0X = j_1i_1X$. Thus, the interpolation property above will again be equivalent to the interpolation property of the above square in the sense of~\cref{interpcat}.
\end{rem}

\begin{rem}[Variations on the notion of interpolation]
    When moving from propositional to first-order logic (from posets to categories), there are several different possibilities of categorifying the notion of interpolation, at least from an algebraic point of view. For instance, one might consider a more literal categorification of the interpolation property for posets into categories
    \[\begin{tikzcd}
    	{\mc C} & {\mc A} \\
    	{\mc B} & {\mc D}
    	\arrow["f", from=1-1, to=1-2]
    	\arrow["g"', from=1-1, to=2-1]
    	\arrow[""{name=0, anchor=center, inner sep=0}, "u", from=1-2, to=2-2]
    	\arrow[""{name=1, anchor=center, inner sep=0}, "v"', from=2-1, to=2-2]
    	\arrow["\alpha", curve={height=-6pt}, between={0.2}{0.8}, Rightarrow, from=1, to=0]
    \end{tikzcd}\]
    by stating that the above lax square has interpolation iff for any $b\in\mc B$ and $a\in\mc A$, given any \emph{morphism} $x : vb \to ua$, there exists $c\in\mc C$ and morphisms $x_0 : b \to gc$ and $x_1 : fc \to a$ making the diagram below commute,
    \[
    \begin{tikzcd}
        vgc \ar[r, "\alpha"] & ufc \ar[d, "ux_1"] \\
        vb \ar[r, "x"'] \ar[u, "vx_0"] & ua
    \end{tikzcd}
    \]
    This notion of interpolation squares for instance appears in~\cite{vcubric1993results}, where the author discusses interpolation for (bi)Cartesian closed categories. However, the intuition there is quite different: Such an interpolation property for certain squares of (bi)Cartesian closed categories are understood as an refinement for the classical interpolation property for intuitionistic \emph{propositional} logic, upgrading it to a \emph{proof relevant} theorem via the Curry-Howard correspondence; also see e.g.~\cite{saurin2025interpolation}. Given~\cref{syntacreffis}, we believe the notion given in~\cref{interpcat} is the more natural choice for the questions addressed by the current paper, but we acknowledge that by no means this is the only interesting notion of interpolation of categories to consider.
\end{rem}

\subsection{Doctrines and interpolation}

One nice thing about the notion of interpolation given in~\cref{interpcat} is that it is a general notion that in principle applies to any type of fragments of logic that admits a classifying category construction. The literature of categorical logic usually refers to such logics with the generic term \textit{doctrine} (\cite{kock1977doctrines}). In this sense, doctrines are a cluster of variations on infinitary predicate logic for which we have an understanding of their syntactic category. This is not a precise definition, but the literature has collected aggregated evidence that a good way to approach these objects is to study lax-idempotent pseudomonads on $\lex$ (\cite{di2025bi}), the $2$-category of small lex categories. Often (and especially in its seminal days) lax-idempotent pseudomonads are \textit{just} called doctrines precisely because of this intuition (\cite{zoberlein1976doctrines}). In this paper we restrict our notion of doctrine to the definition below, we shall comment in a moment about our choice.

\begin{defn}[Doctrines] \label{doctrinesdef}
    Recall from~\cite{di2025logic} that a doctrine is a relative pseudomonad $(\ms T,\eta)$ over the inclusion $j : \lex \to \Lex$ from small lex categories into (locally small) lex categories, equipped with a locally fully faithful pseudonatural transformation $\sigma : \ms T \nt \psh$ between $\ms T$ and the presheaf construction $\psh$. We denote the 2-category of (small) pseudo-algebras for $\ms T$ as $\alg(\ms T)$.
\end{defn}

\begin{war}[Keep it small] \label{boundeddoc}
Throughout the paper we shall restrict our attention to \textit{bounded} doctrines in the sense of \cite[Sec. 5.1.1]{di2025logic}. This means precisely that the relative pseudomonad actually lands in small lex categories, and thus we can avoid every possible size issue. This would not be entirely required but makes the discussion much smoother and allows for crisper proofs.
\end{war}

\begin{rem}[A \textit{reasonable} definition of doctrine]
All the $2$-categories of \textit{theories} associated to fragments of predicate logic can be specified by a doctrine in the sense above (see the introduction of \cite{di2025bi} and its last subsection). It is also useful to see \cite[Example 1.1.8]{di2024sketches}. The $2$-category of -- say -- small coherent categories is indeed the category of algebras for a lax-idempotent pseudomonad on $\lex$ sitting inside the presheaf construction (\cite[Sec. 6]{di2025bi} and \cite[Sec. 5.6]{lex}). Relative pseudomonads are one possible solution to handle size issue: the $2$-category of -- say -- infinitary pretopoi is indeed the 2-category of algebras for the relative lax-idempotent pseudomonad on $\lex$ given by the presheaf construction \cite[Prop. 2.5]{lex}.
\end{rem}

% \begin{rem}
%     While there is in principle ways to mitigate this problem, this notion of doctrine does not contain much information concerning its logical operators {\color{red} I think it contains a lot of info about the "logical operators" (in the sense of conj, disj. What it lacks is their "semantics" properties". Do you disagree?}. {\color{blue} This sentence was written by you... but I agree with you now.} Later we will restrict to doctrines of the form $\ms T^{\mc H}$ induced by a family of geometric morphisms $\mc H$ as constructed in~\cite{di2025logic}. For such doctrines we expect to have a better control of a more \emph{syntactic} understanding of the underlying logical operations, and for all the examples we care in practice are of this form; see \emph{loc. cit.}. 
% \end{rem}

For a doctrine $\ms T$ viewed as a fragment of logic, given~\cref{syntacrefprop} and~\cref{syntacreffis}, the interpolation property for the logic will be equivalent to the interpolation property for certain colimit squares of $\ms T$-algebras. As also mentioned in~\cref{laxsquare}, the appropriate colimit squares will be \emph{cocommas} rather than \emph{pushouts} in the context of doctrines on $\Lex$. Thus we make the following definition:

\begin{defn}[Interpolation property for doctrines]\label{intfordoc}
    Let $\ms T$ be a doctrine on lex categories. We say it has the interpolation property if every cocomma square in $\alg(\ms T)$ has interpolation.
\end{defn}

% {\color{red} there is a sudden jump in the narrative, why suddently move from lax squares to specifically cocomma ones? The remark below answers a different question}

\begin{rem}
    Technically, the statement that every pushout square has the interpolation property is simply false for most of the fragments of logic we care about. For instance, this fails even for the (2-)category of distributive lattices (cf.~\cite{pitts1983amalgamation}).
\end{rem}

\section{Doctrines preserving slicing: a propositional bootstrap} \label{sec:slicing}

In this section we will introduce the notion of \textit{doctrines preserving slicing} and prove that for such doctrines we can prove a \textit{propositional bootstrap}, i.e.\ the interpolation property can be detected by a propositional truncation of the doctrine (\Cref{reducetosub1}).

\subsection{The construction of slices in lex categories}

One operation that will turn out to be extremely crucial is taking \emph{slices} in the category of algebras of a doctrine. However, to give it a 2-categorical description of the slice operation, we need to observe some additional structures present in $\Lex$. Indeed slices are often understood as \textit{comma} objects in $\Cat$ (and thus as certain forms of $2$-dimensional limits), while in $\Lex$ they enjoy colimit-like properties.

\begin{rem}[Remember our friend $\mbb O$?]
Let $\mbb O$ be the free lex category on one generator. As a category it is given by $\Fin\op$, the opposite category of finite sets, but we only care about its universal property rather than its concrete construction. It corepresents the forgetful 2-functor $\Lex \to \Cat$, i.e.\ for any $\mc A\in\Lex$,
\[ \mc A \simeq \Lex(\mbb O,\mc A). \]
Also note that $\Lex$ has a \emph{zero-object}, which is given by the one-object category where we denote as $0$. 
\end{rem}

Let $\mc A$ be a lex category. For any $a\in\mc A$, there is a left exact pullback functor $\mc A \to \mc A/a$, which takes $b\in\mc A$ to the projection $b \times a \to a$. In other words, this is the right adjoint of the forgetful functor $\mc A/a \to \mc A$. The pullback functor gives the slice a mapping out universal property as lex categories:

\begin{lem}[Slices are cocommas]
    For any lex category $\mc A$ and $a\in\mc A$, the pullback functor $\mc A \to \mc A/a$ arises as the following cocomma in $\Lex$,
    \[\begin{tikzcd}
    	{\mbb O} & {\mc A} \\
    	0 & {\mc A/a}
    	\arrow["a", from=1-1, to=1-2]
    	\arrow[from=1-1, to=2-1]
    	\arrow[""{name=0, anchor=center, inner sep=0}, from=1-2, to=2-2]
    	\arrow[""{name=1, anchor=center, inner sep=0}, from=2-1, to=2-2]
    	\arrow[curve={height=-6pt}, between={0.2}{0.8}, Rightarrow, from=1, to=0]
    \end{tikzcd}\]
    The 2-cell is given by the diagonal $\delta_a : 1_a \to \pi_a$ in $\mc A/a$, where $1_a : a \to a$ and $\pi_a : a \times a \to a$ are the images of $1,a$ along $\mc A \to \mc A/a$.
\end{lem}
\begin{proof}
    It is well-known that to give a lex functor $\mc A/a \to \mc E$, viz. a diagram below
    % https://q.uiver.app/#q=WzAsMyxbMCwwLCJcXG1jIEEiXSxbMCwxLCJcXG1jIEEvYSJdLFsxLDAsIlxcbWMgRSJdLFswLDFdLFswLDIsIngiXSxbMSwyLCJcXHBhaXJ7eCxcXGFscGhhfSIsMix7InN0eWxlIjp7ImJvZHkiOnsibmFtZSI6ImRhc2hlZCJ9fX1dXQ==
    \[\begin{tikzcd}
    	{\mc A} & {\mc E} \\
    	{\mc A/a}
    	\arrow["x", from=1-1, to=1-2]
    	\arrow[from=1-1, to=2-1]
    	\arrow["{\pair{x,\alpha}}"', dashed, from=2-1, to=1-2]
    \end{tikzcd}\]
    it is equivalent to give a lex functor $x : \mc A \to \mc E$ and a global section $1 \to xa$ in $\mc E$. This immediately implies the above lax square is a cocomma.
\end{proof}

\begin{rem}[Slicing as coinserter]\label{slicecoinserter}
    Equivalently, the map $\mc A \to \mc A/a$ can also be realised as the following \emph{coinserter}.     In particular, this means the slicing map $\mc A \to \mc A/a$ is an \emph{eso} in the 2-category $\Lex$.
    % https://q.uiver.app/#q=WzAsMyxbMCwwLCJcXG1iYiBPIl0sWzAsMSwiXFxtYyBBIl0sWzEsMCwiQS9hIl0sWzAsMSwiMSIsMCx7Im9mZnNldCI6LTF9XSxbMCwxLCJhIiwyLHsib2Zmc2V0IjoxfV0sWzEsMiwiIiwyLHsiY3VydmUiOjJ9XSxbMCwyLCJcXGlkX2EiLDAseyJvZmZzZXQiOi0yfV0sWzAsMiwiXFxwaV9hIiwyLHsib2Zmc2V0IjoyfV0sWzYsNywiXFxkZWx0YV9hIiwwLHsic2hvcnRlbiI6eyJzb3VyY2UiOjIwLCJ0YXJnZXQiOjIwfX1dXQ==
    \[\begin{tikzcd}
    	{\mbb O} & {A/a} \\
    	{\mc A}
    	\arrow[""{name=0, anchor=center, inner sep=0}, "{\id_a}", shift left=2, from=1-1, to=1-2]
    	\arrow[""{name=1, anchor=center, inner sep=0}, "{\pi_a}"', shift right=2, from=1-1, to=1-2]
    	\arrow["1", shift left, from=1-1, to=2-1]
    	\arrow["a"', shift right, from=1-1, to=2-1]
    	\arrow[curve={height=12pt}, from=2-1, to=1-2]
    	\arrow["{\delta_a}", between={0.2}{0.8}, Rightarrow, from=0, to=1]
    \end{tikzcd}\]
\end{rem}
\begin{rem}[Slicing as a mapping-out universal property]
    In more concrete terms, the above is a coinserter diagram simply means that a lex functor from the slice $\mc A/a$ is concretely given by a pair
    \[ \pair{f,x} : \mc A/a \to \mc B \]
    where $f : \mc A \to \mc B$ is a lex functor, and $x : 1 \to fa$ is a 2-cell, or equivalently a global section of $fa$ in $\mc B$. In this case, the diagram below commutes,
    \[
    \begin{tikzcd}
        \mc A \ar[r, "f"] \ar[d] & \mc B \\ 
        \mc A/X \ar[ur, "\pair{f,x}"']
    \end{tikzcd}
    \]
\end{rem}

\subsection{Doctrines preserving slicing, a propositional bootstrap}

\subsubsection{Doctrines preserving slicing and their algebras}
\begin{defn}\label{preserveslice}
    Given a doctrine $(\ms T,\eta)$, we say it \emph{preserves slicing} if for any $\mc A\in\lex$ and any $a\in\mc A$, the canonical lex functor below is an equivalence,
    \[ (\ms T\mc A)/\eta a \to \ms T(\mc A/a). \]
  
\end{defn}

\begin{exa}
    We will see in~\cref{algslicing} that this property is almost equivalent to the fact that $\ms T$-algebras should be closed under slicing, and they have the usual universal property explained in~\cref{logicslice}. All the known example of doctrines corresponding to first-order logics has this property, including the geometric, coherent, regular, disjunctive fragment, etc.
\end{exa}

\begin{rem}
    Curiously, this basic property in the context of propositional logic also appears in~\cite{sterling2025domains}, where the closure under slicing implies that the universe of $\ms T$-propositions of a fragment of propositional logic $\ms T$ in the classifying topos of $\ms T$-algebras form a \emph{dominance} in the sense of~\cite{rosolini1986continuity}. We feel that the closuring under slicing is such a fundamental property for logics that deserves a better understanding.
\end{rem}

\begin{rem}[Logical interpretation of taking slices]\label{logicslice}
    From a logical perspective, if we identify a $\ms T$-algebra $\mc A$ as the classifying category of some theory $\mbb T$ in the fragment of logic corresponding to $\ms T$, then $\ms T$-algebras being closed under slices corresponds to the fact that \emph{add a (finite family of) constant(s) satisfying some formula $\varphi$} is computed by slicing. Indeed, let $X\in\mc A$ corresponds to some formula $\varphi(\ov x)$ in $\mbb T$. Then the above universal property of the slice $\mc A/X$ exactly states that it corresponds to the new theory $\mbb T/\varphi$ obtained as follows:
    \begin{itemize}
        \item For any free variable $x_i$ in $\varphi(\ov x)$, add a new \emph{constant} $c_i$ of the same sort;
        \item Then also add a new \emph{axiom} stating that $\varphi(\ov c)$ holds.
    \end{itemize}
\end{rem}

Now we proceed to connect the preservation of slicing to the closure of algebras under slicing, this will be achieved in \Cref{algslicing}.
First, we give an alternative characterisation of when a doctrine $\ms T$ preserves slicing via the following notion of \emph{cocartesian arrows}:

\begin{defn}
  For $f : \mc A \to \mc B$ in $\lex$, we say it is \emph{$\ms T$-cocartesian} if the unit square for $f$ is a pseudo-pushout,
  \[\begin{tikzcd}
    \mc A & \mc B \\
    {\ms T\mc A} & {\ms T\mc B}
    \arrow["f", from=1-1, to=1-2]
    \arrow["\eta"', from=1-1, to=2-1]
    \arrow["\eta", from=1-2, to=2-2]
    \arrow[from=2-1, to=2-2]
    \arrow["\lrcorner"{anchor=center, pos=0.125, rotate=180}, draw=none, from=2-2, to=1-1]
  \end{tikzcd}\]
\end{defn}

\begin{lem}\label{preservelocalisationcocartesian}
  $\ms T$ preserves slicing iff for any $\mc A$ and $a\in\mc A$, the lex functor $\mc A \to \mc A/a$ is $\ms T$-cocartesian.
\end{lem}
\begin{proof}
  In the following diagram below,
  % https://q.uiver.app/#q=WzAsNixbMCwwLCJcXG1iYiBPIl0sWzAsMSwiMCJdLFsxLDAsIkEiXSxbMSwxLCJBL2EiXSxbMiwwLCJcXG1zIFRBIl0sWzIsMSwiXFxtcyBUQSJdLFswLDFdLFswLDIsImEiXSxbMSwzXSxbMiwzXSxbMiw0LCJcXHNpZ21hIl0sWzMsNSwiXFxzaWdtYSIsMl0sWzQsNV0sWzgsOSwiIiwwLHsiY3VydmUiOi0xLCJzaG9ydGVuIjp7InNvdXJjZSI6MjAsInRhcmdldCI6MjB9fV1d
    \[\begin{tikzcd}
    	{\mbb O} & A & {\ms TA} \\
    	0 & {A/a} & {\ms TA}
    	\arrow["a", from=1-1, to=1-2]
    	\arrow[from=1-1, to=2-1]
    	\arrow["\eta", from=1-2, to=1-3]
    	\arrow[""{name=0, anchor=center, inner sep=0}, from=1-2, to=2-2]
    	\arrow[from=1-3, to=2-3]
    	\arrow[""{name=1, anchor=center, inner sep=0}, from=2-1, to=2-2]
    	\arrow["\eta"', from=2-2, to=2-3]
    	\arrow[curve={height=-6pt}, between={0.2}{0.8}, Rightarrow, from=1, to=0]
    \end{tikzcd}\]
  by definition $\ms T$ preserves slicing iff the outer square is a cocomma. Since the left square is a cocomma, by the pasting lemma this holds iff the right square is a 2-pushout. 
\end{proof}

We still want to prove that if $\ms T$ is a doctrine preserving slicing, then for any $\ms T$-algebra $\mc A$, the slice $\mc A/a$ is again a $\ms T$-algebra, and $\mc A \to \mc A/a$ is a homomorphism between $\ms T$-algebras. In fact, one can prove a stronger result, which will then entail our claim.

\begin{prop}[$\ms T$-algebras descend along $\ms T$-cocartesian maps]\label{algdescent}
   If a 1-cell $f : \mc A \to \mc B$ is $\ms T$-cocartesian and $\mc A$ is a $\ms T$-algebra, then so is $\mc B$, and $f$ is a $\ms T$-algebra morphism. 
\end{prop}
\begin{proof}
  Suppose $\gamma_{\mc A} \dashv \eta_{\mc A} : \ms T\mc A \to \mc A$ makes $\mc A$ a $\ms T$-algebra. Consider the following diagram,
  % https://q.uiver.app/#q=WzAsNSxbMCwwLCJcXG1jIEEiXSxbMSwwLCJcXG1jIEIiXSxbMCwxLCJcXG1zIFRcXG1jIEEiXSxbMSwxLCJcXG1zIFRcXG1jIEIiXSxbMiwxLCJcXG1jIEIiXSxbMCwxLCJmIl0sWzAsMiwiXFxldGFfe1xcbWMgQX0iXSxbMSwzLCJcXGV0YV97XFxtYyBCfSJdLFsxLDQsIiIsMCx7ImxldmVsIjoyLCJzdHlsZSI6eyJoZWFkIjp7Im5hbWUiOiJub25lIn19fV0sWzIsMCwiXFxnYW1tYV97XFxtYyBBfSIsMCx7ImN1cnZlIjotMn1dLFsyLDMsIlxcbXMgVGYiXSxbMiw0LCJ7ZlxcZ2FtbWFfQX0iLDIseyJjdXJ2ZSI6M31dLFszLDQsIlxcZ2FtbWFfe1xcbWMgQn0iLDEseyJzdHlsZSI6eyJib2R5Ijp7Im5hbWUiOiJkYXNoZWQifX19XSxbOSw2LCIiLDAseyJsZXZlbCI6MSwic3R5bGUiOnsibmFtZSI6ImFkanVuY3Rpb24iLCJib2R5Ijp7Im5hbWUiOiJub25lIn0sImhlYWQiOnsibmFtZSI6Im5vbmUifX19XV0=
    \[\begin{tikzcd}
    	{\mc A} & {\mc B} \\
    	{\ms T\mc A} & {\ms T\mc B} & {\mc B}
    	\arrow["f", from=1-1, to=1-2]
    	\arrow[""{name=0, anchor=center, inner sep=0}, "{\eta_{\mc A}}", from=1-1, to=2-1]
    	\arrow["{\eta_{\mc B}}", from=1-2, to=2-2]
    	\arrow[equals, from=1-2, to=2-3]
    	\arrow[""{name=1, anchor=center, inner sep=0}, "{\gamma_{\mc A}}", curve={height=-12pt}, from=2-1, to=1-1]
    	\arrow["{\ms Tf}", from=2-1, to=2-2]
    	\arrow["{{f\gamma_A}}"', curve={height=18pt}, from=2-1, to=2-3]
    	\arrow["{\gamma_{\mc B}}"{description}, dashed, from=2-2, to=2-3]
    	\arrow["\dashv"{anchor=center}, draw=none, from=1, to=0]
    \end{tikzcd}\]
  This shows we have a map $\gamma_{\mc B} : \ms T\mc B \to \mc B$ with 
  \[ \gamma_{\mc B} \eta_{\mc B} \cong \id, \quad \gamma_{\mc B}\ms Tf \cong f\gamma_{\mc A}, \]
  thus if $\gamma_{\mc B} \dashv \eta_{\mc B}$ then $\mc B$ will be an $\ms T$-algebra and $f : \mc A \to \mc B$ will be a $\ms T$-algebra morphism. For this, we consider the 2-dimensional universal property of the pushout square, and observe we have two maps $\id,\eta_{\mc B}\gamma_{\mc B} : \ms T\mc B \to \ms T\mc B$. Furthermore, since $\gamma_{\mc A} \dashv \eta_{\mc A}$, there are 2-cells 
  \[ \id \circ \ms Tf \cong \ms Tf \nt \ms Tf \circ \eta_{\mc A}\gamma_{\mc A} \cong \eta_{\mc B} f \gamma_{\mc A} \cong \eta_{\mc B}\gamma_{\mc B} \circ \ms Tf, \]
  and
  \[ \id \circ \eta_{\mc B} \cong \eta_{\mc B} \cong \eta_{\mc B}\gamma_{\mc B} \circ \eta_{\mc B}. \]
  Thus, we get a unique 2-cell $\id \nt \eta_{\mc B}\gamma_{\mc B}$. This 2-cell satisfies the triangle equality again by the universal property of the 2-pushout.
\end{proof}

The upshot is that if $\ms T$ preserves slicing, then $\ms T$-algebras will be closed under the slicing operation, and furthermore, the same universal property holds:

\begin{prop} \label{algslicing}
    If $\ms T$ preserves slicing, then for any $\ms T$-algebra $\mc A$ and $a\in\mc A$, $\mc A/a$ is again a $\ms T$-algebra where $\mc A \to \mc A/a$ is a $\ms T$-algebra morphism. Furthermore, the following is also a cocomma in $\alg(\ms T)$,
    \[\begin{tikzcd}
    	{\ms T\mbb O} & {\mc A} \\
    	{\ms T0} & {\mc A/a}
    	\arrow["\qsi a", from=1-1, to=1-2]
    	\arrow[from=1-1, to=2-1]
    	\arrow[""{name=0, anchor=center, inner sep=0}, from=1-2, to=2-2]
    	\arrow[""{name=1, anchor=center, inner sep=0}, from=2-1, to=2-2]
    	\arrow[curve={height=-6pt}, between={0.2}{0.8}, Rightarrow, from=1, to=0]
    \end{tikzcd}\]
    where here $\qsi a : \ms T\mbb O \to \mc A$ is the left Kan extension of $a : \mbb O \to \mc A$ along $\eta : \mbb O \hook \ms T\mbb O$.
\end{prop}
\begin{proof}
    The first half of the statement is a corollary of~\cref{preservelocalisationcocartesian} and~\cref{algdescent}. To show the above square is a cocomma, suppose we have a $\ms T$-algebra morphism $f : \mc A \to \mc X$ with a 2-cell $x : 1 \to fa$. It suffices to show the uniquely determined map $\pair{f,x} : \mc A/a \to \mc X$ below is a $\ms T$-algebra morphism,
    % https://q.uiver.app/#q=WzAsMyxbMCwwLCJcXG1jIEEiXSxbMSwwLCJYIl0sWzAsMSwiXFxtYyBBL2EiXSxbMCwxLCJmIl0sWzAsMiwibCIsMl0sWzIsMSwiXFxwYWlye2YseH0iLDIseyJzdHlsZSI6eyJib2R5Ijp7Im5hbWUiOiJkYXNoZWQifX19XV0=
    \[\begin{tikzcd}
    	{\mc A} & X \\
    	{\mc A/a}
    	\arrow["f", from=1-1, to=1-2]
    	\arrow["l"', from=1-1, to=2-1]
    	\arrow["{\pair{f,x}}"', dashed, from=2-1, to=1-2]
    \end{tikzcd}\]
    For this we observe that $\ms T$-preserves slicing implies that there is an equivalence $\ms T(\mc A/a) \simeq \ms T\mc A/\eta a$, and under this equivalence,
    \[ \ms T\pair{f,x} \cong \pair{\ms Tf,\ms Tx} : \ms T(\mc A/a) \simeq \ms T\mc A/\eta a \to \mc X. \]
    Thus, it follows that 
    \[ \gamma_{\mc X}\ms T\pair{f,x} \cong \gamma_{\mc X} \circ \pair{\ms Tf,\ms Tx} \cong \pair{\gamma_{\mc X}\ms Tf,\gamma_{\mc X}\ms Tx} \cong \pair{f\gamma_{\mc A},x\gamma_{\mbb O}} \cong \pair{f,x} \circ \gamma_{\mc A/a} \]
    This implies $\pair{f,x}$ will be a $\ms T$-algebra morphism.
\end{proof}

\begin{rem}\label{slicepushout}
    Suppose $\ms T$ preserves slicing. For any $f : \mc A \to \mc B$ in $\alg(\ms T)$ and $a\in\mc A$, by the universal property above, the following will be a 2-pushout in $\alg(\ms T)$.
    % https://q.uiver.app/#q=WzAsNCxbMCwwLCJcXG1jIEEiXSxbMSwwLCJcXG1jIEIiXSxbMCwxLCJcXG1jIEEvWCJdLFsxLDEsIlxcbWMgQi9mWCJdLFswLDEsImYiXSxbMCwyXSxbMSwzXSxbMiwzLCJmL1giLDJdLFszLDAsIiIsMSx7InN0eWxlIjp7Im5hbWUiOiJjb3JuZXIifX1dXQ==
    \[\begin{tikzcd}
    	{\mc A} & {\mc B} \\
    	{\mc A/a} & {\mc B/fa}
    	\arrow["f", from=1-1, to=1-2]
    	\arrow[from=1-1, to=2-1]
    	\arrow[from=1-2, to=2-2]
    	\arrow["{f/X}"', from=2-1, to=2-2]
    	\arrow["\lrcorner"{anchor=center, pos=0.125, rotate=180}, draw=none, from=2-2, to=1-1]
    \end{tikzcd}\]
\end{rem}

\subsubsection{A propositional bootstrap}
Finally we can prove a first instange of a \textit{propositional bootstrap}: we observe that the interpolation property for a doctrine $\ms T$ on $\Lex$ preserving slicing can be reduced to the interpolation of subobject lattices on $1$.

\begin{prop}[Propositional bootstrap]\label{reducetosub1}
    Let $\ms T$ be a doctrine on lex categories which preserves slicing. Then it has the interpolation property iff for any cocomma square in $\alg(\ms T)$, the image under the 2-functor $\sub_{-}(1)$ has interpolation in the sense of~\cref{interpos}.
\end{prop}
\begin{proof}
    The only if direction is evident: Consider a cocomma square in $\alg(\ms T)$,
   % https://q.uiver.app/#q=WzAsNCxbMCwwLCJDIl0sWzAsMSwiRCJdLFsxLDAsIkEiXSxbMSwxLCJCIl0sWzAsMiwiZiJdLFswLDEsImciLDJdLFsxLDMsInYiLDJdLFsyLDMsInUiXSxbNiw3LCJcXGFscGhhIiwwLHsiY3VydmUiOi0xLCJzaG9ydGVuIjp7InNvdXJjZSI6MjAsInRhcmdldCI6MjB9fV1d
    \[\begin{tikzcd}
    	\mc C & \mc A \\
    	\mc D & \mc B
    	\arrow["f", from=1-1, to=1-2]
    	\arrow["g"', from=1-1, to=2-1]
    	\arrow[""{name=0, anchor=center, inner sep=0}, "u", from=1-2, to=2-2]
    	\arrow[""{name=1, anchor=center, inner sep=0}, "v"', from=2-1, to=2-2]
    	\arrow["\alpha", curve={height=-6pt}, between={0.2}{0.8}, Rightarrow, from=1, to=0]
    \end{tikzcd}\]
    When $X\in\mc C$ is taken to be the terminal object $1$, \cref{interpcat} exact expresses that its image under $\sub_{-}(1)$ has interpolation.

    For the if direction, given a cocomma square as above, for any $X\in\mc C$ consider the following diagram,
    % https://q.uiver.app/#q=WzAsOCxbMCwwLCJcXG1jIEMiXSxbMiwwLCJcXG1jIEEiXSxbMCwyLCJcXG1jIEIiXSxbMiwyLCJcXG1jIEQiXSxbMSwxLCJcXG1jIEMvWCJdLFszLDEsIlxcbWMgQS9mWCJdLFsxLDMsIlxcbWMgQi9nWCJdLFszLDMsIlxcbWMgRC92Z1giXSxbMCwxLCJmIl0sWzAsMiwiZyIsMl0sWzEsMywidSIsMSx7ImxhYmVsX3Bvc2l0aW9uIjoyMH1dLFsyLDMsInYiLDEseyJsYWJlbF9wb3NpdGlvbiI6MjB9XSxbNCw1XSxbMCw0XSxbMSw1XSxbNCw2XSxbMiw2XSxbNiw3XSxbNSw3XSxbMyw3XSxbMTEsMTAsIlxcYWxwaGEiLDEseyJzaG9ydGVuIjp7InNvdXJjZSI6MjAsInRhcmdldCI6MjB9fV0sWzE3LDE4LCJcXGFscGhhIiwxLHsibGFiZWxfcG9zaXRpb24iOjQwLCJzaG9ydGVuIjp7InNvdXJjZSI6MjAsInRhcmdldCI6MjB9fV1d
    \[\begin{tikzcd}
    	{\mc C} && {\mc A} \\
    	& {\mc C/X} && {\mc A/fX} \\
    	{\mc B} && {\mc D} \\
    	& {\mc B/gX} && {\mc D/vgX}
    	\arrow["f", from=1-1, to=1-3]
    	\arrow[from=1-1, to=2-2]
    	\arrow["g"', from=1-1, to=3-1]
    	\arrow[from=1-3, to=2-4]
    	\arrow[""{name=0, anchor=center, inner sep=0}, "u"{description, pos=0.2}, from=1-3, to=3-3]
    	\arrow[from=2-2, to=2-4]
    	\arrow[from=2-2, to=4-2]
    	\arrow[""{name=1, anchor=center, inner sep=0}, from=2-4, to=4-4]
    	\arrow[""{name=2, anchor=center, inner sep=0}, "v"{description, pos=0.2}, from=3-1, to=3-3]
    	\arrow[from=3-1, to=4-2]
    	\arrow[from=3-3, to=4-4]
    	\arrow[""{name=3, anchor=center, inner sep=0}, from=4-2, to=4-4]
    	\arrow["\alpha"{description}, between={0.2}{0.8}, Rightarrow, from=2, to=0]
    	\arrow["\alpha"{description, pos=0.4}, between={0.2}{0.8}, Rightarrow, from=3, to=1]
    \end{tikzcd}\]
    By~\cref{slicepushout}, both the top and bottom squares are pushouts in $\alg(\ms T)$. Thus the front square is again a cocomma. Now it is easy to see that the interpolation property expressed in~\cref{interpcat} for $X\in\mc C$ exactly says that the image of the front lax square under $\sub_{-}(1)$ has interpolation.
\end{proof}

\section{Finitary doctrines, conservative maps and interpolation} \label{sec:finitary}

The goal of this section is to prove that \textit{finitary} doctrines that preserve slicing admitting interpolation have a very satisfactory characterization (\Cref{tstableinterpolation}) in terms of truth conservative maps (\Cref{tcons}). This is very much in the same spirit of traditional algebraic logic investigation of interpolation, showing that the interpolation property of a fragment of logic will be equivalent to a certain structural property in the corresponding category of algebras. Even more, our analysis mirrors quite closely the correspondence between interpolation and the so called \textit{amalgamation property} of the class of algebras. We discuss this connection in \Cref{hoog}.

\subsection{Finitary doctrines, filter quotients, and $t$-conservative maps}

Let $\ms T$ be a doctrine on lex categories preserving slicing. Recall from \cite{di2025bi} that such a doctrine is called \emph{finitary} when $\ms T$ preserves filtered pseudocolimits. In this section we shall assume to work with finitary doctrines. In this case, we can construct an \emph{orthogonal factorisation system} in $\alg(\ms T)$. The left class will be (a subclass of) \emph{localisations} as defined below, while the right class will be given by truth conservative maps, which we shall define later.

\begin{defn}[Localisation]
    Let $\ms T$ be a finitary doctrine closed under slices. Let $\mc A\in\alg(\ms T)$. A \emph{localisation} of $\mc A$ at a cofiltered diagram $f : \mc I \to \mc A$ is the following filtered colimit in $\alg(\ms T)$,
    \[ \mc A_f := \ct_{i\in\mc I}\mc A/fi. \]
    The finitary assumption on $\ms T$ makes sure this filtered pseudocolimit in $\alg(\ms T)$ exists, and is created by the forgetful functor into $\Cat$ (or equivalently $\Lex$).
\end{defn}

The above notion of localisation unifies a range of different constructions.

\begin{exa}[Slices are localisations]
    By definition, for any $\mc A\in\alg(\ms T)$ and $X\in\mc A$, the slice $\mc A \to \mc A/X$ is a localisation of $\mc A$. As mentioned in~\cref{logicslice}, this corresponds to adding a finite family of constants satisfying certain formula.
\end{exa}

\begin{exa}[Filter quotients are localisations]
    Let $F$ be a filter on the meet-semi-lattice $\sub_{\mc A}(1)$. This induces a localisation
    \[ \mc A \surj \mc A_F \simeq \ct_{u\in F}\mc A/u. \]
    Logically, this corresponds to the quotient theory of adding all the closed formulas in $F$ as axioms.
\end{exa}

\begin{exa}[Localisation at a structure/point]\label{exa:localisationatpoint}
    Let $M : \mc A \to \Set$ be a left exact functor. We have a projection $\elem M \to \mc A$ where $\elem M$ is the category of elements of $M$, and is cofiltered since $M$ is left exact. We define the \emph{localisation of $\mc A$ at $M$} as follows,
    \[ \mc A_M \simeq \ct_{(a,x) \in (\elem M)\op} \mc A/a. \]
    Alternatively, consider the pseudofunctor
    $ \mc A/- : \mc A\op \to \alg(\ms T).$
    Then one can show that $\mc A_M$ is the $M$-weighted colimit of $\mc A/-$, which means that for any $\mc E \in \alg(\ms T)$,
    \[ \alg(\ms T)(\mc A_M,\mc E) \simeq [\mc A,\Cat](M,\alg(\ms T)(\mc A/-,\mc E)). \]
\end{exa}

\begin{rem}
    For pretopoi or topoi, the special case of localisation at a \emph{model} $M : \mc A \to \Set$, in the sense that it is a $\ms T$-algebra homormophism, has been considered in the literature~\cite{johnstone1989local,caramello2023over,breiner2014scheme}. Logically, this corresponds to the so called \emph{theory of diagram} of a model (see e.g.~\cite{chang1990model}), i.e.\ add a constant for each element in the model, together with all closed formulas satisfied by the model as axioms.
\end{rem}

\begin{rem}
    Note that the slices are a special case of localisation at a structure, since $\mc A/X \simeq \mc A_{\yo^X}$ for the corepresentable functor $\yo^X = \mc A(X,-) : \mc A \to \Set$.
\end{rem}

The right class we are going to consider are those maps satisfying certain weak conservativity properties:

\begin{defn} \label{tcons}
    Let $f : \mc A \to \mc B$ be a map in $\Lex$. We say it is \emph{truth conservative}, or \emph{t-conservative} in short, if it reflects truth of subterminal object, i.e.\ for any $u \in \sub_{\mc A}(1)$, $fu = 1$ iff $u = 1$.
\end{defn}

\begin{rem}\label{fcons}
    If the doctrine $\ms T$ also provides a bottom element $0$, then we may say a map $f : A \to b$ is \emph{falsum conservative}, or \emph{f-conservative}, if it reflects $0$.
\end{rem}

\begin{lem}\label{tconmono}
    In a diagram below, if $f$ is t-conservative, then so is $g$:
    \[
    \begin{tikzcd}
        \mc A \ar[dr, "f"] \ar[d, "g"'] \\
        \mc B \ar[r, "h"'] & \mc C
    \end{tikzcd}
    \]
\end{lem}
\begin{proof}
    This is trivial.
\end{proof}

\begin{prop}\label{factorisation}
    For a finitary doctrine $\ms T$ preserving slicing, the pair (filter quotients, t-conservative maps) in $\alg(\ms T)$ forms an orthogonal factorisation system.
\end{prop}
\begin{proof}
    Given any map $f : \mc A \to \mc B$ in $\alg(\ms T)$, we may factor it as below,
    \[
    \begin{tikzcd}
        \mc A \ar[dr, "q"'] \ar[rr, "f"] & & \mc B \\ 
        & \mc A/{f\inv(1)} \ar[ur, "i"']
    \end{tikzcd}
    \]
    where $f\inv(1)$ is the filter on $\sub_{\mc A}(1)$ that are mapped to $1$ by $f$. By construction, $\mc A \to \mc A/f\inv(1)$ is a filter quotient, and $A/{f\inv(1)} \to B$ is t-conservative.

    Now consider a square as below where $\mc A \to \mc A/F$ is a filter quotient at $F$ and $g$ is t-conservative,
    % https://q.uiver.app/#q=WzAsNCxbMCwwLCJBIl0sWzAsMSwiQS9GIl0sWzEsMCwiQiJdLFsxLDEsIkMiXSxbMCwxLCIiLDIseyJzdHlsZSI6eyJoZWFkIjp7Im5hbWUiOiJlcGkifX19XSxbMCwyLCJmIl0sWzIsMywiZyJdLFsxLDMsImgiLDJdLFsxLDIsIiIsMSx7InN0eWxlIjp7ImJvZHkiOnsibmFtZSI6ImRhc2hlZCJ9fX1dXQ==
    \[\begin{tikzcd}
    	\mc A & \mc B \\
    	{\mc A/F} & \mc C
    	\arrow["f", from=1-1, to=1-2]
    	\arrow[two heads, from=1-1, to=2-1]
    	\arrow["g", from=1-2, to=2-2]
    	\arrow[dashed, from=2-1, to=1-2]
    	\arrow["h"', from=2-1, to=2-2]
    \end{tikzcd}\]
    To show the dashed diagonal morphism exists in $\alg(\ms T)$, by the universal property of the quotient $\mc A/F$, it suffices to show that $fu = 1$ for all $u\in F$. But we now that $gfu = 1$ for all $u\in F$ since $gf$ factors through $\mc A/F$. Since $g$ is t-conservative, it follows that $fu = 1$ for all $u\in F$.
\end{proof}

\begin{rem}[A closely related orthogonal factorisation system] \label{Axel}
    \cite[Ch. 10]{osmond2021categorical} has constructed a closely related orthogonal factorisation system on $\Lex$. In particular, the construction in~\cref{exa:localisationatpoint} has appeared in \emph{loc. cit.\ }, under the name of \emph{focalisation}. Osmond shows the existence of an orthogonal factorisation system (focalisation, terminally connected) on $\Lex$, where a lex functor is terminally connected iff it \emph{lifts global sections functorially}. The factorisation system appeared in~\cref{factorisation} can be seen as a \emph{propositional counterpart}: instead of all localisations (focalisations) we only consider the \emph{filter} quotients, which are filtered colimits of slices of \emph{subterminal objects}; and instead of considering terminally connected functors, we consider t-conservative ones which only lifts global sections of \emph{subterminal objects}. The reason we do this is due to the propositional bootstrap result presented in~\cref{reducetosub1}. \cite{osmond2021categorical} has also observed that the same factorisation system lifts to one on the 2-categories of regular and coherent categories, but~\cref{factorisation} provides a structural reason for this: almost the same proof can be used to show that if a doctrine $\ms T$ preserves slices, then it also has a (localisation, terminally connected) orthogonal factorisation system.
\end{rem}

\subsection{Finitary doctrines closed under slices and interpolation}

At the end of this section, our goal is to show that for any finitary doctrine $\ms T$ preserving slicing, the interpolation property of $\ms T$ can be equivalently expressed as certain \emph{exactness property} of $\ms T$-algebras.

\begin{defn}
    We say a class of maps $\mc M$ is \emph{closed under cocommas} in $\alg(\ms T)$, if given any cocomma square as below, if $g\in\mc M$ then $u \in \mc M$.
    % https://q.uiver.app/#q=WzAsNCxbMCwwLCJcXG1jIEMiXSxbMSwwLCJcXG1jIEEiXSxbMCwxLCJcXG1jIEIiXSxbMSwxLCJcXG1jIEQiXSxbMCwxLCJmIl0sWzAsMiwiZyIsMl0sWzEsMywidSJdLFsyLDMsInYiLDJdLFs3LDYsIiIsMix7ImN1cnZlIjotMSwic2hvcnRlbiI6eyJzb3VyY2UiOjIwLCJ0YXJnZXQiOjIwfX1dXQ==
    \[\begin{tikzcd}
    	{\mc C} & {\mc A} \\
    	{\mc B} & {\mc D}
    	\arrow["f", from=1-1, to=1-2]
    	\arrow["g"', from=1-1, to=2-1]
    	\arrow[""{name=0, anchor=center, inner sep=0}, "u", from=1-2, to=2-2]
    	\arrow[""{name=1, anchor=center, inner sep=0}, "v"', from=2-1, to=2-2]
    	\arrow[curve={height=-6pt}, between={0.2}{0.8}, Rightarrow, from=1, to=0]
    \end{tikzcd}\]
\end{defn}

\begin{rem}
    Since cocomma squares have a direction, there are two potential notions for a class of maps to be closed under cocomma. We might say $\mc M$ is \emph{op-closed} under cocommas if in a cocomma square as above, when $f\in\mc M$ then $u\in\mc M$.
\end{rem}

\begin{thm}\label{tstableinterpolation}
    Let $\ms T$ be a finitary doctrine on lex categories preserving slicing. It has the interpolation property iff t-conservative maps are closed under cocomma in $\alg(\ms T)$.
\end{thm}
\begin{proof}
    For the only if direction, consider a cocomma square in $\alg(\ms T)$ below,
    % https://q.uiver.app/#q=WzAsNCxbMCwwLCJcXG1jIEMiXSxbMSwwLCJcXG1jIEEiXSxbMCwxLCJcXG1jIEIiXSxbMSwxLCJcXG1jIEQiXSxbMCwxLCJmIl0sWzAsMiwiZyIsMl0sWzEsMywidSJdLFsyLDMsInYiLDJdLFs3LDYsIlxcYWxwaGEiLDAseyJjdXJ2ZSI6LTEsInNob3J0ZW4iOnsic291cmNlIjoyMCwidGFyZ2V0IjoyMH19XV0=
    \[\begin{tikzcd}
    	{\mc C} & {\mc A} \\
    	{\mc B} & {\mc D}
    	\arrow["f", from=1-1, to=1-2]
    	\arrow["g"', from=1-1, to=2-1]
    	\arrow[""{name=0, anchor=center, inner sep=0}, "u", from=1-2, to=2-2]
    	\arrow[""{name=1, anchor=center, inner sep=0}, "v"', from=2-1, to=2-2]
    	\arrow["\alpha", curve={height=-6pt}, between={0.2}{0.8}, Rightarrow, from=1, to=0]
    \end{tikzcd}\]
    Suppose $g$ is t-conservative. To show $u$ is t-conservative, consider any $a\in\sub_{\mc A}(1)$ that $ua = 1$. Then $v1 \le ua$, thus by the interpolation property we must have $c\in\sub_{\mc C}(1)$ that $1 \le gc$ and $fc \le a$. Now since $g$ is t-conservative, $gc = 1$ implies $c = 1$. Hence, $fc = 1 \le a$, which means $a = 1$. This shows $u$ is also t-conservative.

    For the if direction, we use the reduction to subobject lattices on 1 discussed in~\cref{reducetosub1}. Again consider a cocomma square in $\alg(\ms T)$ as before, and suppose we have $b\in \sub_{\mc B}(1)$ and $a\in\sub_{\mc A}(1)$ that $vb \le ua$ in $\sub_{\mc D}(1)$. Consider the following diagram,
    % https://q.uiver.app/#q=WzAsOCxbMCwwLCJcXG1jIEMiXSxbMiwwLCJcXG1jIEEiXSxbMSwxLCJcXG1jIEMvZ1xcaW52KGIpIl0sWzMsMSwiXFxtYyBBL2ZnXFxpbnYoYikiXSxbMCwyLCJcXG1jIEIiXSxbMiwyLCJcXG1jIEQiXSxbMSwzLCJcXG1jIEIvYiJdLFszLDMsIlxcbWMgRC92YiJdLFswLDEsImYiXSxbMCwyLCIiLDAseyJzdHlsZSI6eyJoZWFkIjp7Im5hbWUiOiJlcGkifX19XSxbMCw0LCJnIiwyXSxbMSwzLCIiLDAseyJzdHlsZSI6eyJoZWFkIjp7Im5hbWUiOiJlcGkifX19XSxbMSw1LCJ1IiwxLHsibGFiZWxfcG9zaXRpb24iOjIwfV0sWzIsM10sWzIsNl0sWzQsNSwidiIsMSx7ImxhYmVsX3Bvc2l0aW9uIjoyMH1dLFs0LDYsIiIsMCx7InN0eWxlIjp7ImhlYWQiOnsibmFtZSI6ImVwaSJ9fX1dLFs1LDcsIiIsMCx7InN0eWxlIjp7ImhlYWQiOnsibmFtZSI6ImVwaSJ9fX1dLFs2LDcsInYiLDFdLFszLDcsInUiXSxbMTUsMTIsIlxcYWxwaGEiLDEseyJzaG9ydGVuIjp7InNvdXJjZSI6MjAsInRhcmdldCI6MjB9fV0sWzE4LDE5LCIiLDEseyJzaG9ydGVuIjp7InNvdXJjZSI6MjAsInRhcmdldCI6MjB9fV1d
    \[\begin{tikzcd}
    	{\mc C} && {\mc A} \\
    	& {\mc C/g\inv(b)} && {\mc A/fg\inv(b)} \\
    	{\mc B} && {\mc D} \\
    	& {\mc B/b} && {\mc D/vb}
    	\arrow["f", from=1-1, to=1-3]
    	\arrow[two heads, from=1-1, to=2-2]
    	\arrow["g"', from=1-1, to=3-1]
    	\arrow[two heads, from=1-3, to=2-4]
    	\arrow[""{name=0, anchor=center, inner sep=0}, "u"{description, pos=0.2}, from=1-3, to=3-3]
    	\arrow[from=2-2, to=2-4]
    	\arrow[from=2-2, to=4-2]
    	\arrow[""{name=1, anchor=center, inner sep=0}, "u", from=2-4, to=4-4]
    	\arrow[""{name=2, anchor=center, inner sep=0}, "v"{description, pos=0.2}, from=3-1, to=3-3]
    	\arrow[two heads, from=3-1, to=4-2]
    	\arrow[two heads, from=3-3, to=4-4]
    	\arrow[""{name=3, anchor=center, inner sep=0}, "v"{description}, from=4-2, to=4-4]
    	\arrow["\alpha"{description}, between={0.2}{0.8}, Rightarrow, from=2, to=0]
    	\arrow[between={0.2}{0.8}, Rightarrow, from=3, to=1]
    \end{tikzcd}\]
    By~\cref{slicepushout}, again both the top and bottom squares are pushout in $\alg(\ms T)$. Since the back square is a cocomma in $\alg(\ms T)$, so is the front square. Now by construction, the map $\mc C/g\inv(\dv b) \to \mc B/b$ is t-conservative, thus so is the map $\mc A/fg\inv(\dv b) \to\mc D/vb$. Given this, since $vb \le ua$, which means $ua = 1$ in $\mc D/vb$, we must already have $ua = 1$ in $\mc A/fg\inv(\dv b)$. This means there exists $c\in g\inv(\dv b)$, viz. $b \le gc$, that $fc \le a$. Hence, the subobject lattices on $1$ in the original square has interpolation. This implies $\ms T$ has the interpolation property by~\cref{reducetosub1}.
\end{proof}

There are two important aspects of~\cref{tstableinterpolation} compared to the traditional of algebraic logic. First is that it is a categorification of results about propositional logic into the first-order case:

\begin{rem}[A comparison with Algebraic Logic] \label{hoog}    It is indeed an important observation in algebraic logic that certain \emph{exactness properties} of the corresponding doctrine of propositional algebras is intimately connected to various logical properties of the propositional logic, with interpolation being one prominent example. For an \emph{implicative fragment}, i.e.\ a fragment where implication is a primitive logical operator, one typically finds an equivalence between the interpolation property of the logic and the corresponding category of algebras having the so-called \emph{strong amalgamation property}, which categorically is simply stating that monomorphisms will be stable under pushout; see e.g.~\cite{hoogland2001definability,galli2003strong}. Thus, \cref{tstableinterpolation} above is a similar statement stating the equivalence between the interpolation property of a fragment of logic with the stability of a class of ``weak monomorphisms'' which we call t-conservative maps.
\end{rem}

However, there is also an obvious difference between our consideration of the ``lax'' squares and cocommas rather than commuting squares and pushouts:

\begin{rem}[Why lax diagrams?]
    For a purely \emph{positive} fragment of logic, where implication is not a primitive operation, considering \emph{lax properties} is necessary, and this applies to both propositional and first-order logic. For instance, for the propositional case, it is observed in~\cite{pitts1983amalgamation} that the category of distributive lattices satisfies the strong amalgamation property, i.e.\ monomorphisms of distributive lattices are indeed closed under pushouts. However, pushout squares of distributive lattices fail to have the interpolation property in general. Instead, it will be a consequence of our main result in~\cref{interpolationofsubfragment} and~\cref{interpolationpretopoi} that cocomma squares of distributive lattices indeed have the interpolation property. 
\end{rem}

\begin{rem}[Limitations of this work] \label{limitations}
    In particular, this means that the techniques in our paper only directly applies to fragments of geometric logic, which exclude e.g. first-order intuitionistic logic. It is possible to have a broader picture for a more general notion of doctrine, say a suitable 2-monad on the \emph{(2,1)-category} of left exact categories with natural \emph{isomorphisms}. The reason one cannot work fully 2-categorically is due to the fact that, in the presence of negative operators, say implication and universal quantifier, the corresponding doctrine \emph{cannot} be 2-monadic over $\Lex$: The functor taking a lex category to the free Heyting category simply cannot be extended to a 2-functor --- the same already happens when one consider the free Heyting algebra associated to a meet-semi-lattice. 
\end{rem}

\begin{rem}[A comment on Beth definability]
Finally, we shall observe that when a logic has a \textit{deduction theorem}, one can often derive a form of Beth definability from the presence of Craig interpolation. Also this wisdom emerges very clearly from algebraic logic (see the discussion in \cite[3.3-3.6]{hoogland2001definability}). Unfortunately geometric logic does not have an implication among its defining features, and lacks a deduction theorem, hence this paper does not try to prove a form of Beth definability as a corollary of the efforts put in order to prove Craig interpolation. In order to discuss a variation of geometric logic that takes into account an implication, one should work in the $2$-category of topoi and \textit{open} geometric morphisms and re-develop Kan injectivity there. There is some evidence in the literature that this would be a somewhat pathological $2$-category \cite{bezhanishvili2024category}.
\end{rem}

\section{From doctrines to logics: from slicing to \'etale map classifiers} \label{sec:etale}
Our next step is to switch the point of view of the paper from \textit{doctrines} (\Cref{doctrinesdef}) to (fragments of geometric) logic in the sense of \cite{di2025logic}. The notion of \textit{logic} was introduced in \cite{di2025logic} having in mind different semantic prescriptions, and we refer to Sec 4 of that paper for a complete contexualization. Here we provide a brief recap on the set up:

The aim of this section will be to establish which fragments of geometric logic will correspond to doctrines preserving slicing, so that we can apply~\cref{tstableinterpolation} in the next section. Interestingly, this will correspond to the existence of a classifier of \'etale maps (\Cref{etaleimpliessclicing}).

\subsection{A recap on fragments of geometric logic}

A fragment of geometric logic in the sense of~\cite{di2025logic} is specified by a family of geometric morphisms $\mc H$. This induces a sub 2-category $\WRInj(\mc H)$ of $\Topoi$, where objects are those topoi $\mc X$ where for any $f : \mc E \to \mc F$ in $\mc H$ and any $x : \mc E \to \mc X$, 
\[
\begin{tikzcd}
    \mc E \ar[r, "x"] \ar[d, "f"'] & \mc X \\ 
    \mc F \ar[ur, dashed, "\ran_fx"']
\end{tikzcd}
\]
the right Kan extension $\ran_fx$ in the 2-category $\Topoi$ of topoi exists. In this case we say $\mc X$ is \emph{weakly right Kan injective} w.r.t. $\mc H$. A morphism in $\WRInj(\mc H)$ is a geometric morphism that preserves these right Kan extensions. 

Any family $\mc H$ induces a doctrine $\ms T^{\mc H}$, where for any $\mc A\in\lex$, the free algebra $\ms T^{\mc H}(\mc A)$ is given as follows,
\[ \ms T^{\mc H}(\mc A) \simeq \WRInj(\mc H)(\psh(\mc A),\Set[\mbb O]), \]
where here $\Set[\mbb O]$ is the object classifier. More generally, for any $\mc X\in\WRInj(\mc H)$, we define its \emph{syntactic category} as follows,
\[ \syn^{\mc H}(\mc X) \simeq \WRInj(\mc X,\Set[\mbb O]). \]

\begin{war}[Bounded logics]
As we previously did for the case of doctrines (\Cref{boundeddoc}), we shall restrict to study only \textit{bounded} logics in the sense of \cite[Sec. 5.1.1]{di2025logic}.
\end{war}

\begin{rem}
    The reason the above is well-defined is because that free topoi, viz. presheaf topoi over lex categories, are weakly right Kan injective w.r.t. \emph{all} geometric morphisms. It is instructive to recall how to compute such right Kan extensions (cf.~\cite{diliberti2022geometry,di2025logic}): Given $f : \mc E \to \mc F$ and $x : \mc E \to \psh(\mc A)$ with $\mc A$ left exact, 
    \[
    \begin{tikzcd}
        \mc E \ar[d, "f"'] \ar[r, "x"] & \psh(\mc A) \\
        \mc F \ar[ur, dashed, "\ran_fx"']
    \end{tikzcd}
    \]
    the right Kan extension $\ran_fx$ above can be computed as
    \[ (\ran_fx)^* \cong \lan_{\yo} f_*x^*\yo, \quad (\ran_fx)_* \cong \lan_{f_*}x_*. \]
\end{rem}

\begin{rem}\label{syncat}
    Concretely, a presheaf $Y$ in $\psh(\mc A)$ belongs to $\ms T^{\mc H}\mc A$ iff for any $f : \mc E \to \mc F$ in $\mc H$ and $x : \mc E \to \psh(\mc A)$, the lower triangle below commutes,
    \[
    \begin{tikzcd}
        \mc E \ar[d, "f"'] \ar[r, "x"] & \psh(\mc A) \ar[d, "Y"] \\
        \mc F \ar[r, dashed, "\ran_fYx"'] \ar[ur, dashed, "\ran_fx"description] & \Set[\mbb O]
    \end{tikzcd}
    \]
    Since right Kan extensions into $\Set[\mbb O]$ computes direct images (cf.~\cite[Exm. 1.3.2]{di2025logic}), concretely this exactly means that
    \[ (\ran_fx)^*Y \cong f_*x^*Y. \]
    More specifically, if we write $Y \cong \ct_{i\in I}A_i$ as a colimit of representables, the above requires that
    \[ \ct_{i\in I}f_*x^*A_i \cong f_*x^*\ct_{i\in I}A_i. \]
    A sufficient condition for this to hold is that $f_*$ preserves colimit indexed by $I$.
\end{rem}

\subsection{\'Etale map classifier}
In this subsection we shall discuss the notion of \'etale map classifier for a logic $\mathcal{H}$. The existence of such classifier will entail that the associated monad $\mathsf{T}^\mathcal{H}$ preserves slicing (\Cref{etaleimpliessclicing}).
% {\color{red} talk about classifiers in general, if it exists, it must be this one anyway}

Let us start by observing that we have a \textit{classifier for \'etale maps} in $\Topoi$. Let $\Set[\mbb O]$ be the object classifier and we also use $\mbb O$ to denote the universal object in $\Set[\mbb O]$. Then any \'etale map is a pullback as below,
% https://q.uiver.app/#q=WzAsNCxbMCwwLCJcXHBzaChcXG1jIEEpL1giXSxbMSwwLCJcXFNldFtcXG1iYiBPXS9cXG1iYiBPIl0sWzAsMSwiXFxwc2goXFxtYyBBKSJdLFsxLDEsIlxcU2V0W1xcbWJiIE9dIl0sWzAsMV0sWzAsMiwie3tcXFBpX1h9fSIsMl0sWzEsM10sWzIsMywiWCIsMl0sWzAsMywiIiwxLHsic3R5bGUiOnsibmFtZSI6ImNvcm5lciJ9fV1d
\[\begin{tikzcd}
    {\mc X/X} & {\Set[\mbb O]/\mbb O} \\
    {\mc X} & {\Set[\mbb O]}
    \arrow[from=1-1, to=1-2]
    \arrow["{{{\Pi_X}}}"', from=1-1, to=2-1]
    \arrow["\lrcorner"{anchor=center, pos=0.125}, draw=none, from=1-1, to=2-2]
    \arrow[from=1-2, to=2-2]
    \arrow["X"', from=2-1, to=2-2]
\end{tikzcd}\]

We may similarly define what it means to have an \'etale map classifier for an arbitrary fragment of logic $\mc H$:

\begin{defn}\label{absoluteetale}
    We say a logic $\mc H$ has an \emph{\'etale map classifier}, if there is an \'etale map $\mc Z_* \to \mc Z$ in $\WRInj(\mc H)$, such that for any \'etale map $\mc Y \to \mc X$ in $\WRInj(\mc H)$ there is an essentially unique map $\mc X \to \mc Z$ making the following a pullback, 
    \[
    \begin{tikzcd}
        \mc Y \ar[d] \ar[r] & \mc Z_* \ar[d] \\ 
        \mc X \ar[r] & \mc Z
        \arrow["\lrcorner"{anchor=center, pos=0.125}, draw=none, from=1-1, to=2-2]
    \end{tikzcd}
    \]
\end{defn}

\begin{rem}
    Since the forgetful functor $\WRInj(\mc H) \to \Topoi$ creates and preserves 2-limits, being a pullback in $\WRInj(\mc H)$ and in $\Topoi$ means the same.
\end{rem}

The observation is that, if a logic $\mc H$ has an \'etale map classifier, it must be the one in $\Topoi$:

\begin{prop}
    If $\mc H$ has an \'etale map classifier, then it must be equivalent to $\Set[\mbb O]/\mbb O \to \Set[\mbb O]$.
\end{prop}
\begin{proof}
    Suppose $\mc Z_* \to \mc Z$ is the \'etale map classifier for $\mc H$. By definition this is \'etale, thus we get an essentially unique map $f : \mc Z \to \Set[\mbb O]$ where the following is a pullback,
    \[
    \begin{tikzcd}
        \mc Z_* \ar[d] \ar[r] & \Set[\mbb O]/\mbb O \ar[d] \\ 
        \mc Z \ar[r, "f"'] & \Set[\mbb O]
        \arrow["\lrcorner"{anchor=center, pos=0.125}, draw=none, from=1-1, to=2-2]
    \end{tikzcd}
    \]
    On the other hand, $\Set[\mbb O]$ is a \emph{free topos}, and the universal object $\mbb O$ is \emph{representable} in $\Set[\mbb O]$. This implies the generic \'etale map $\Set[\mbb O]/\mbb O \to \Set[\mbb O]$ in $\Topoi$ belongs to $\WRInj(\mc H)$ for any $\mc H$. Thus, by the universal property of $\mc Z_* \to \mc Z$, we also get an essentially unique map $g : \Set[\mbb O] \to \mc Z$ with a pullback
    \[
    \begin{tikzcd}
        \Set[\mbb O]/\mbb O \ar[d] \ar[r] & \mc Z_* \ar[d] \\ 
        \Set[\mbb O] \ar[r, "g"'] & \mc Z
        \arrow["\lrcorner"{anchor=center, pos=0.125}, draw=none, from=1-1, to=2-2]
    \end{tikzcd}
    \]
    By the respective universal property of $\Set[\mbb O]/\mbb O \to \Set[\mbb O]$ and $\mc Z_* \to \mc Z$, we must have $fg \cong \id$ and $gf \cong \id$, thus $\mc Z_* \to \mc Z$ must be equivalent to $\Set[\mbb O]/\mbb O \to \Set[\mbb O]$.
\end{proof}

\begin{rem}
    Recall that for a coherent topos $\mc X$, for any $X\in\mc X$ the topos $\mc X/X$ is coherent iff $X$ is a coherent object in $\mc X$. The above condition is an abstraction of this fact that applies to an arbitrary fragment of logic.
\end{rem}

\begin{prop}\label{etaleimpliessclicing}
    If a fragment $\mc H$ has an \'etale map classifier, then $\ms T^{\mc H}$ preserves slicing.
\end{prop}
\begin{proof}
    For any $\mc A\in\Lex$ and $X\in\mc A$, we have a canonical equivalence 
    \[ \psh(\mc A/X) \simeq \psh(\mc A)/X. \]
    The canonical functor $\ms T^{\mc H}(\mc A)/X \to \ms T^{\mc H}(\mc A/X)$ lives over this equivalence,
    % https://q.uiver.app/#q=WzAsNCxbMCwwLCJcXG1zIFRee1xcbWMgSH1cXG1jIEEvWCJdLFswLDEsIlxccHNoKFxcbWMgQSkvWCJdLFsxLDAsIlxcbXMgVF57XFxtYyBIfShcXG1jIEEvWCkiXSxbMSwxLCJcXHBzaChcXG1jIEEvWCkiXSxbMCwxLCIiLDAseyJzdHlsZSI6eyJ0YWlsIjp7Im5hbWUiOiJob29rIiwic2lkZSI6InRvcCJ9fX1dLFsxLDMsIlxcc2ltZXEiLDJdLFsyLDMsIiIsMCx7InN0eWxlIjp7InRhaWwiOnsibmFtZSI6Imhvb2siLCJzaWRlIjoidG9wIn19fV0sWzAsMl1d
    \[\begin{tikzcd}
    	{\ms T^{\mc H}\mc A/X} & {\ms T^{\mc H}(\mc A/X)} \\
    	{\psh(\mc A)/X} & {\psh(\mc A/X)}
    	\arrow[from=1-1, to=1-2]
    	\arrow[hook, from=1-1, to=2-1]
    	\arrow[hook, from=1-2, to=2-2]
    	\arrow["\simeq"', from=2-1, to=2-2]
    \end{tikzcd}\]
    which shows it is fully faithful. Thus, to show it is an equivalence it suffices to show it is essentially surjective. Equivalently, suppose we have $y : Y \to X$ in $\ms T^{\mc H}(\mc A/X) \subseteq \psh(\mc A/X) \simeq \psh(\mc A)/X$, we have to show that $Y\in\ms T^{\mc H}(\mc A)$. However, we have the following commutative diagram,
    % https://q.uiver.app/#q=WzAsNCxbMSwxLCJcXHBzaChcXG1jIEEpIl0sWzEsMCwiXFxwc2goXFxtYyBBKS9YIl0sWzAsMCwiKFxccHNoKFxcbWMgQSkvWCkveSJdLFswLDEsIlxccHNoKFxcbWMgQSkvWSJdLFsxLDAsIlxcUGlfWCJdLFszLDIsIlxcc2ltZXEiXSxbMiwxLCJcXFBpX3kiXSxbMywwLCJcXFBpX1kiLDJdXQ==
    \[\begin{tikzcd}
    	{(\psh(\mc A)/X)/y} & {\psh(\mc A)/X} \\
    	{\psh(\mc A)/Y} & {\psh(\mc A)}
    	\arrow["{\Pi_y}", from=1-1, to=1-2]
    	\arrow["{\Pi_X}", from=1-2, to=2-2]
    	\arrow["\simeq", from=2-1, to=1-1]
    	\arrow["{\Pi_Y}"', from=2-1, to=2-2]
    \end{tikzcd}\]
    Since $X$ is representable, $\Pi_X\in\WRInj(\mc H)$; $\Pi_y \in \WRInj(\mc H)$ since $y\in\ms T^{\mc H}(\mc A/X)$ by assumption. It follows that the composite, 
    \[ \Pi_Y \cong \Pi_X\Pi_y : \psh(\mc A)/Y \simeq (\psh(\mc A)/X)/y \to \psh(\mc A) \] 
    belongs to $\WRInj(\mc H)$ as well. By the fact that $\mc H$ has an \'etale map classifier, it follows that $Y\in\ms T^{\mc H}(\mc A)$. Thus $\ms T^{\mc H}$ preserves slicing. 
\end{proof}

\begin{rem}
    For the above proof it suffices for $\mc H$ to have absolute \'etale classifier for \emph{free topoi}, not necessarily for arbitrary $\mc X\in\WRInj(\mc H)$ as in~\cref{absoluteetale}. 
\end{rem}

\section{Interpolation for finitary logics} \label{sec:intfin}

The remaining goal of this section is to show that all finitary logics between \emph{regular logic} and \emph{coherent logic}, which has an \'etale classifier, will have the interpolation property. By~\cref{tstableinterpolation}, it suffices to show t-conservative maps of $\ms T^{\mc H}$-algebras will be closed under cocommas. 

Our strategy is to exploit the classifying topoi construction for a logic described in~\cite{di2025logic}. Recall that the classifying topos construction for any logic $\mc H$ produces a relative adjoint to the forgetful functors from $\Topoi$ to $\ms{Alg}(\ms T^{\mc H})$,
% https://q.uiver.app/#q=WzAsMyxbMSwwLCJcXG1hdGhzZntUb3BvaX0iXSxbMCwwLCJcXG1hdGhzZntBbGd9KFxcbWF0aHNme1R9XntcXEhjYWx9KV97XFxtYXRoc2Z7TX19XFxvcCJdLFswLDEsIlxcbWF0aHNme0FsZ30oXFxtYXRoc2Z7VH1ee1xcSGNhbH0pXFxvcCJdLFsxLDAsIlxcbXN7Q2x9Il0sWzAsMl0sWzEsMl0sWzQsMywiIiwyLHsibGV2ZWwiOjEsInN0eWxlIjp7Im5hbWUiOiJhZGp1bmN0aW9uIn19XV0=
\[\begin{tikzcd}
    {\mathsf{alg}(\mathsf{T}^{\Hcal})} & {\mathsf{Topoi}\op} \\
    {\mathsf{Alg}(\mathsf{T}^{\Hcal})}
    \arrow[""{name=0, anchor=center, inner sep=0}, "{\ms{Cl}}", from=1-1, to=1-2]
    \arrow[from=1-1, to=2-1]
    \arrow[""{name=1, anchor=center, inner sep=0}, from=1-2, to=2-1]
    \arrow["\dashv"{anchor=center, rotate=-90}, draw=none, from=1, to=0]
\end{tikzcd}\]
The crucial observation is that we can detect t-conservativity of maps between algebras on the level of the corresponding geometric morphism of their classifying topoi. We call the relevant class of geometric morphism \emph{op-dominant}:

\begin{defn}
    A geometric morphism $f : \mc X \to \mc Y$ is \emph{op-dominant} if the inverse image $f^*$ is t-conservative.
\end{defn}

\begin{rem}
    Note that a geometric morphism is called \emph{dominant} if $f^*$ reflect 0, i.e.\ $f^*$ is f-conservative. This is why we choose the above terminology.
\end{rem}

Recall from~\cite{di2025logic} that the regular fragment corresponds to the class of \emph{matte} geometric morphisms $\mc H_{\mr{matte}}$, which are those geometric morphisms whose direct images preserve epimorphisms. Similarly, the coherent fragment corresponds to the class of \emph{flat} geometric morphisms $\mc H_{\mr{flat}}$, which are those geometric morphisms whose direct images preserve epimorphisms and binary coproducts. 

\begin{prop}\label{tconsopdom}
    If $\mc H_{\mr{matte}} \subseteq \mc H$, then for a map $f : \mc A \to \mc B$ in $\alg(\ms T^{\mc H})$, $f$ is t-conservative iff $\Cl[f]$ is op-dominant.
\end{prop}
\begin{proof}
    Consider the following commuting diagram,
    % https://q.uiver.app/#q=WzAsNixbMCwwLCJBIl0sWzEsMCwiXFxDbFtBXSJdLFswLDEsIkIiXSxbMSwxLCJcXENsW0JdIl0sWzIsMCwiXFxtYyBQKEEpIl0sWzIsMSwiXFxtYyBQKEIpIl0sWzAsMiwiZiIsMl0sWzAsMSwiIiwwLHsic3R5bGUiOnsidGFpbCI6eyJuYW1lIjoibW9ubyJ9fX1dLFsxLDMsIlxcQ2xbZl1eKiIsMV0sWzIsMywiIiwyLHsic3R5bGUiOnsidGFpbCI6eyJuYW1lIjoibW9ubyJ9fX1dLFs0LDUsIlxcbWMgUChmKSJdLFsxLDQsIiIsMSx7InN0eWxlIjp7InRhaWwiOnsibmFtZSI6Im1vbm8ifX19XSxbMyw1LCIiLDEseyJzdHlsZSI6eyJ0YWlsIjp7Im5hbWUiOiJtb25vIn19fV1d
    \[\begin{tikzcd}
    	\mc A & {\Cl[\mc A]} & {\psh(\mc A)} \\
    	\mc B & {\Cl[\mc B]} & {\psh(\mc B)}
    	\arrow[tail, from=1-1, to=1-2]
    	\arrow["f"', from=1-1, to=2-1]
    	\arrow[tail, from=1-2, to=1-3]
    	\arrow["{\Cl[f]^*}"{description}, from=1-2, to=2-2]
    	\arrow["{\psh(f)}", from=1-3, to=2-3]
    	\arrow[tail, from=2-1, to=2-2]
    	\arrow[tail, from=2-2, to=2-3]
    \end{tikzcd}\]
    If $\Cl[f]$ is op-dominant, i.e.\ $\Cl[f]^*$ is t-conservative, then by~\cref{tconmono} so is $f$. On the other hand, notice that $\mc H_{\mr{matte}} \subseteq \mc H$ implies that there is a forgetful functor 
    \[ \ms{alg}(\ms T^{\mc H}) \to \ms{alg}(\ms T^{\mc H_{\mr{matte}}}), \]
    and in particular any $\ms T^{\mc H}$-algebra will be regular. If $f$ is t-conservative, then so is $\psh(f)$: For a regular category $\mc A$, we have an isomorphism 
    \[ \sub_{\psh(\mc A)}(1) \cong \ms{D}(\sub_{\mc A}(1)), \]
    where $\ms D(-)$ takes the downward closed sets of a meet-semi-lattice. It is straightforward to see that the functor $\ms{D}(-)$ preserves t-conservative maps between meet-semi-lattices. This shows that if $f$ is t-conservative then so is $\psh(f)$, thus $\Cl[f]$ is op-dominant by~\cref{tconmono} again.
\end{proof}

We will use the above characterisation to show that the property of t-conservative maps being closed under cocomma is \emph{inherited by subfragment of logic}. 

\begin{lem}\label{freeprest}
    Let $\mc H_0 \subseteq \mc H_1$ be two logics. If t-conservative maps are closed under cocommas in $\alg(\ms T^{\mc H_1})$, then so is the case for $\mc H_0$. 
\end{lem}
\begin{proof}
    By~\cite[Thm. 7.7]{lex} we get a free functor $F : \alg(\ms T^{\mc H_0}) \to \alg(\ms T^{\mc H_{1}})$ be the free functor. We claim that $F$ preserves t-conservative maps. Let $\mc A\in\alg(\ms T^{\mc H_0})$. We note that $\mc A$ and $F\mc A$ has the same classifying topos. This follows straightforwardly by the universal property: For any topos $\mc E$, we have the following chain of equivalences,
    \[ \Topoi(\mc E,\Cl_0[\mc A]) \simeq \alg(\ms T^{\mc H_0})(\mc A,\mc E) \simeq \alg(\ms T^{\mc H_1})(F\mc A,\mc E) \simeq \Topoi(\mc X,\Cl_1[F\mc A]). \]
    Hence the claim holds by~\cref{tconsopdom}. Now consider a cocomma square in $\alg(\ms T^{\mc H_0})$ below,
    % https://q.uiver.app/#q=WzAsOCxbMCwwLCJcXG1jIEEiXSxbMiwwLCJcXG1jIEMiXSxbMSwxLCJGXFxtYyBBIl0sWzMsMSwiRlxcbWMgQyJdLFswLDIsIlxcbWMgQiJdLFsyLDIsIlxcbWMgRCJdLFsxLDMsIkZcXG1jIEIiXSxbMywzLCJGXFxtYyBEIl0sWzAsMV0sWzAsMiwiIiwwLHsic3R5bGUiOnsidGFpbCI6eyJuYW1lIjoibW9ubyJ9fX1dLFswLDRdLFsxLDMsIiIsMCx7InN0eWxlIjp7InRhaWwiOnsibmFtZSI6Im1vbm8ifX19XSxbMSw1XSxbMiwzXSxbMiw2XSxbMyw3XSxbNCw1XSxbNCw2LCIiLDAseyJzdHlsZSI6eyJ0YWlsIjp7Im5hbWUiOiJtb25vIn19fV0sWzUsNywiIiwwLHsic3R5bGUiOnsiYm9keSI6eyJuYW1lIjoiZGFzaGVkIn19fV0sWzYsN10sWzE2LDEyLCIiLDAseyJzaG9ydGVuIjp7InNvdXJjZSI6MjAsInRhcmdldCI6MjB9fV0sWzE5LDE1LCIiLDAseyJzaG9ydGVuIjp7InNvdXJjZSI6MjAsInRhcmdldCI6MjB9fV1d
    \[\begin{tikzcd}
    	{\mc A} && {\mc C} \\
    	& {F\mc A} && {F\mc C} \\
    	{\mc B} && {\mc D} \\
    	& {F\mc B} && {F\mc D}
    	\arrow[from=1-1, to=1-3]
    	\arrow[tail, from=1-1, to=2-2]
    	\arrow[from=1-1, to=3-1]
    	\arrow[tail, from=1-3, to=2-4]
    	\arrow[""{name=0, anchor=center, inner sep=0}, from=1-3, to=3-3]
    	\arrow[from=2-2, to=2-4]
    	\arrow[from=2-2, to=4-2]
    	\arrow[""{name=1, anchor=center, inner sep=0}, from=2-4, to=4-4]
    	\arrow[""{name=2, anchor=center, inner sep=0}, from=3-1, to=3-3]
    	\arrow[tail, from=3-1, to=4-2]
    	\arrow[tail, from=3-3, to=4-4]
    	\arrow[""{name=3, anchor=center, inner sep=0}, from=4-2, to=4-4]
    	\arrow[between={0.2}{0.8}, Rightarrow, from=2, to=0]
    	\arrow[between={0.2}{0.8}, Rightarrow, from=3, to=1]
    \end{tikzcd}\]
    We may apply the left adjoint $F$, where the unit will an embedding. Since $F$ is a left adjoint, the front lax square will be a cocomma square in $\alg(\ms T^{\mc H_1})$. Now if the map $\mc A \to \mc B$ is t-conservative, then by~\cref{freeprest} so is $F\mc A \to F\mc B$. By assumption, so is $F\mc C \to F\mc D$. Thus by~\cref{tconmono}, so is the map $\mc C \to\mc D$. This shows that t-conservative maps are also closed under cocomma in $\alg(\ms T^{\mc H_0})$.
\end{proof}

\begin{rem}[Failure of t-conservative maps being closed under cocomma in geometric logic] \label{cocommatcons}
    We adopt an example mentioned in~\cite{pitts1983amalgamation}, which was attributed to P.T. Johnstone. Consider $X = \N \cup \set{\infty}$ with the following topology: $U$ is open in $X$ iff $U = \emptyset$ or it is cofinite and contains $\infty$. Now let $i : \N \inj X$ be the inclusion with $\N$ as a discrete space. Consider the following cocomma square in $\Frm$,
    % https://q.uiver.app/#q=WzAsNCxbMCwwLCJcXHJne1h9Il0sWzAsMSwiXFxtYyBQKFxcTikiXSxbMSwwLCJcXG1jIFAoWCkiXSxbMSwxLCJcXG1jIFAoXFxOKSJdLFswLDIsIiIsMCx7InN0eWxlIjp7InRhaWwiOnsibmFtZSI6Im1vbm8ifX19XSxbMCwxLCJpXioiLDIseyJzdHlsZSI6eyJ0YWlsIjp7Im5hbWUiOiJtb25vIn19fV0sWzEsMywiIiwwLHsibGV2ZWwiOjIsInN0eWxlIjp7ImhlYWQiOnsibmFtZSI6Im5vbmUifX19XSxbMiwzLCJpXFxpbnYiXV0=
    \[\begin{tikzcd}
    	{\rg{X}} & {\mc P(X)} \\
    	{\mc P(\N)} & {\mc P(\N)}
    	\arrow[tail, from=1-1, to=1-2]
    	\arrow["{i^*}"', tail, from=1-1, to=2-1]
    	\arrow["{i\inv}", from=1-2, to=2-2]
    	\arrow[equals, from=2-1, to=2-2]
    \end{tikzcd}\]
    In this case this is also a pushout, since the poset of points on $X$ is discrete. Furthermore, $i^*$ is t-conservative (and f-conservative) since it is injective. However, $i\inv$ is not t-conservative (nor f-conservative). 
\end{rem}

\begin{thm}\label{interpolationofsubfragment}[Craig interpolation for (existential) finitary subgeometric logics]
    Let $\mc H_{\mr{matte}} \subseteq \mc H \subseteq \mc H_{\mr{flat}}$ be a logic between the regular and coherent fragment, and suppose $\mc H$ has absolute \'etale classifier. Then $\mc H$ has the interpolation property.
\end{thm}
\begin{proof}
    If $\mc H \subseteq \mc H_{\mr{flat}}$, then $\ms T^{\mc H} \subseteq \ms T^{\mc H_{\mr{flat}}}$ is a submonad of the pretopos completion. By~\cite[Lem. 6.4.4]{di2025bi}, $\ms T^{\mc H}$ is finitary. Thus by~\cref{tstableinterpolation}, it suffices to show t-conservative maps of $\ms T^{\mc H}$-algebras are closed under cocommas. By~\cref{freeprest}, it suffices to observe this for $\ms T^{\mc H_{\mr{flat}}}$. This follows from the fact that the doctrine of pretopoi is finitary, and has the interpolation property (see ~\cite{pitts2020interpolation} or the appendix of this paper).
\end{proof}

\begin{rem}
    The proof that the doctrine of pretopoi has the interpolation property in~\cite{pitts2020interpolation} is based on the \emph{topos of filters} construction, which is quite special to the coherent fragment. In~\cref{newintforcoh} we will provide an alternative proof based on classifying topoi of pretopoi alone.
\end{rem}

\begin{exa}[Regular logic with falsum]
    One simple example that~\cref{interpolationofsubfragment} applies to is the fragment of regular logic with falsum. Syntactically, it has all the logical operators of regular logic, plus $\bot$. As a logic, following~\cite{di2025logic} this is described by the class of geometric morphisms whose direct image preserves both the initial object and epimorphisms. This corresponds to the doctrine of exact categories with a strict initial object. It is evident that this doctrine preserves slicing, thus~\cref{interpolationofsubfragment} shows it has the interpolation property.
\end{exa}

\begin{exa}[Regular logic with duplications]\label{duplication}
    We can construct an interesting toy example as a doctrine on lex categories. From the perspective of lex-colimits~\cite{lex}, one may consider the doctrine of freely adding the lex-colimits of the coproduct of the terminal object with itself, $1+1$. Since for any lex category $\mc C$ and $X\in\mc C$, we have the following pullback in the presheaf category
    \[
    \begin{tikzcd}
        X + X \ar[d] \ar[r] & 1+1 \ar[d] \\
        X \ar[r] & 1
        \arrow["\lrcorner"{anchor=center, pos=0.125}, draw=none, from=1-1, to=2-2]
    \end{tikzcd}
    \]
    thus this doctrine also adds the coproduct of any object with itself, and from this it is easy to see this doctrine must be closed under slices. However, if we take the union of this lex-colimits with the one generating exact categories, we do not currently know whether this gives us the full doctrine of pretopoi or not.
\end{exa}

\begin{rem}[A boring example] \label{trivialexa}
    We can also consider the union of disjunctive logic with regular logic. However, this provides nothing new: as a doctrine on lex categories, if an exact category also has finite coproducts, it becomes a pretopos, thus this gives us nothing but the coherent fragment. From the perspective of logic in the sense of~\cite{di2025logic}, viz. as a class of geometric morphisms, it is also evident that a map is flat iff it is both matte and pure.
\end{rem}

\begin{rem}[How many other examples?] \label{lim1}
    How many fragments exist between regular and coherent? After \Cref{trivialexa}, the reader may feel like there aren't that many, and indeed even the examples we bring do not seem that rich or interesting. Another interesting observation is that, if we look at \cref{duplication} from the perspective of Kan injectivity in $\Topoi$, a geometric morphism is pure, i.e.\ its direct image preserves finite coproducts, iff it preserves the specific coproduct $1+1$ (see~\cite[C3.4.12]{elephant2}). Thus, it seems this example again gives us the coherent fragment. Unfortunately, we do not know the exact answer. We currently truly lack a classification of fragments of geometric logic from any perspective, algebraic or proof theoretic. The authors of the paper have already discussed the relevance of this point in \cite[7.3]{di2025logic}. We look at this paper as an additional motivation -- or even \textit{a call} -- to develop that subproject.
\end{rem}

\begin{rem}[What about subregular or non-regular logics?] \label{lim2}
The main theorem of this section (\Cref{interpolationofsubfragment}) requires our logics to sit above the expressivity of regular logic to prove Craig interpolation. Crucially, the proof relies on a technical proposition (\Cref{tconsopdom}) which needs the algebras for $\mathsf{T}^{\mathcal{H}}$ to be regular. \textit{What happens for \textit{subregular} or non-regular\footnote{The disjunctive one, for example.} fragments of geometric logic?} While \Cref{tconsopdom} is needed to reduce $t$-conservative maps to op-dominant geometric morphisms, and thus features in a key portion of the proof, the rest of our strategy is very sound. So, we do not exactly know what happens below or parallelly to regular logic, but vast portions of the technology that we have developed still applies to that context, and future research could focus on investigating how to circumvent or build around \Cref{tconsopdom} to discern Craig interpolation for such fragements of geometric logic.
\end{rem}

\begin{rem}[A modular interpolation: drilling to the minimal common fragment]
    In hindsight, our interpolation result (\cref{interpolationofsubfragment}) should be seen as a \emph{modular} development of interpolation across different fragments of logic. And even for the readers who are only working with first-order logic, our result still provides genuinely new information. The usual interpolation property for first-order logic says if we have $\varphi$ in signature $\Sigma_1$ and $\psi$ in signature $\Sigma_2$ such that $\varphi \vdash \psi$, then we can find an interpolant that belongs to the common signature $\Sigma_1 \cap \Sigma_2$. Our result improves this by showing that if $\varphi,\psi$ belongs to some \emph{subfragment}, then the interpolant can also be found in the \emph{same} subfragment. Given the examples and discussions in~\cite{di2025logic}, this in particular applies to the regular fragment, regular with a bottom element added, etc..
\end{rem}

\appendix

\section{A new proof of interpolation for coherent logic}\label{newintforcoh}

We have mentioned that it is shown in~\cite{pitts2020interpolation} that the doctrine of pretopoi has the interpolation property, though the proof there uses the topos of filter construction. In this section we show that it suffices to look at the classifying topos alone.

Recall the notion of dominant geometric morphisms:

\begin{defn}
    A geometric morphism $f : \mc X \to \mc Y$ is \emph{dominant} if $f_*$ reflect 0, i.e.\ $f^*$ is f-conservative. Equivalently, $f_*$ preserves 0.
\end{defn}

\begin{rem}\label{properdominantsurj}
    Surjectivity always implies being dominant. The other direction holds when $f$ is \emph{proper} (cf.~\cite{moerdijk2000proper}).
\end{rem}

\begin{defn}
    A geometric morphism is \emph{spartan}\footnote{The terminology is due to~\cite{diliberti2022geometry};~\cite{moerdijk2000proper} call such maps \emph{relative tidy}.} if its direct image preserves filtered colimits.
\end{defn}

\begin{prop}\label{sparBC}
    Let below be a comma square in $\Topoi$ where $f$ is spartan. If $f$ is dominant, then so is $v$ (in fact $v$ will be a surjection).
    % https://q.uiver.app/#q=WzAsNCxbMCwwLCJcXG1jIFciXSxbMSwwLCJcXG1jIFgiXSxbMCwxLCJcXG1jIFkiXSxbMSwxLCJcXG1jIFoiXSxbMCwxLCJ1Il0sWzAsMiwidiIsMl0sWzEsMywiZiJdLFsyLDMsImciLDJdLFs3LDYsIiIsMix7ImN1cnZlIjotMSwic2hvcnRlbiI6eyJzb3VyY2UiOjIwLCJ0YXJnZXQiOjIwfX1dXQ==
    \[\begin{tikzcd}
    	{\mc W} & {\mc X} \\
    	{\mc Y} & {\mc Z}
    	\arrow["u", from=1-1, to=1-2]
    	\arrow["v"', from=1-1, to=2-1]
    	\arrow[""{name=0, anchor=center, inner sep=0}, "f", from=1-2, to=2-2]
    	\arrow[""{name=1, anchor=center, inner sep=0}, "g"', from=2-1, to=2-2]
    	\arrow[curve={height=-6pt}, between={0.2}{0.8}, Rightarrow, from=1, to=0]
    \end{tikzcd}\]
\end{prop}
\begin{proof}
    By~\cite{moerdijk2000proper}, when $f$ is spartan, then $v$ is tidy, and this square satisfies the Beck-Chevalley condition $g^*f_* \cong v_*u^*.$
    Thus, if $f$ is dominant, then $f_*0 \cong 0$, and we have, 
    \[ v_*0 \cong v^*u^*0 \cong g^*f_*0 \cong 0. \]
    This implies that $v_*$ is dominant. In fact, since $v$ is tidy, by~\cref{properdominantsurj} it is a surjection.
\end{proof}

The above result can be used to show pretopoi have the interpolation property. The trick is to observe that there is a certain duality between the limit and colimit structure that makes the following dual form of~\cref{tstableinterpolation} holds for pretopoi:

\begin{prop}\label{fconsinter}
    If f-conservative maps are \emph{op-closed} under cocommas in $\DL$, then every cocomma square has interpolation.
\end{prop}

We will not explicitly write down a proof here since it is essentially the same as~\cref{tstableinterpolation}; also cf.~\cite[Rem. 1.11]{pitts2020interpolation}. We observe that f-conservative maps of pretopoi corresponds exactly to dominant geometric morphisms via the classifying topos construction:

\begin{lem}\label{fconsCL}
    For a map $f : \mc D \to \mc E$ of pretopoi, $f$ is f-conservative iff the corresponding geometric morphism $\Cl[f] : \Cl[\mc E] \to \Cl[\mc D]$ is dominant.
\end{lem}
\begin{proof}
    Recall that the direct image $\Cl[f]_*$ is given by precomposition with $f$. Thus, for any $d\in\mc D$, we have
    \[ \Cl[f]_*(0)(d) \cong \mc E(fd,0), \]
    which is empty iff $fd$ is non-trivial. It follows that $\Cl[f]_*(0) \cong 0$ iff $fd \cong 0$ implies $d \cong 0$, i.e.\ $f$ is f-conservative.
\end{proof}

\begin{thm}\label{interpolationpretopoi}
    The doctrine of pretopoi has the interpolation property.
\end{thm}
\begin{proof}
   By~\cref{fconsinter}, it suffices to show f-conservative maps are op-closed under cocommas. Suppose we have a cocomma square with $f$ f-conservative,
   % https://q.uiver.app/#q=WzAsNCxbMCwwLCJcXG1jIEMiXSxbMSwwLCJcXG1jIEEiXSxbMCwxLCJcXG1jIEIiXSxbMSwxLCJcXG1jIEQiXSxbMCwxLCJmIl0sWzAsMiwiZyIsMl0sWzEsMywidSJdLFsyLDMsInYiLDJdLFs3LDYsIiIsMix7ImN1cnZlIjotMSwic2hvcnRlbiI6eyJzb3VyY2UiOjIwLCJ0YXJnZXQiOjIwfX1dXQ==
    \[\begin{tikzcd}
    	{\mc C} & {\mc A} \\
    	{\mc B} & {\mc D}
    	\arrow["f", from=1-1, to=1-2]
    	\arrow["g"', from=1-1, to=2-1]
    	\arrow[""{name=0, anchor=center, inner sep=0}, "u", from=1-2, to=2-2]
    	\arrow[""{name=1, anchor=center, inner sep=0}, "v"', from=2-1, to=2-2]
    	\arrow[curve={height=-6pt}, between={0.2}{0.8}, Rightarrow, from=1, to=0]
    \end{tikzcd}\]
    Then since the functor $\Cl[-]$ is a left adjoint (cf.~\cite{di2025logic}), it takes a cocomma square to a comma square in $\Topoi$ below,
    % https://q.uiver.app/#q=WzAsNCxbMCwwLCJcXENsW1xcbWMgRF0iXSxbMSwwLCJcXENsW1xcbWMgQV0iXSxbMCwxLCJcXENsW1xcbWMgQl0iXSxbMSwxLCJcXENsW1xcbWMgRF0iXSxbMCwxLCJ7XFxDbFt1XX0iXSxbMCwyLCJ7XFxDbFt2XX0iLDJdLFsxLDMsIntcXENsW2ZdfSJdLFsyLDMsIntcXENsW2ddfSIsMl0sWzcsNiwiIiwyLHsiY3VydmUiOi0xLCJzaG9ydGVuIjp7InNvdXJjZSI6MjAsInRhcmdldCI6MjB9fV1d
    \[\begin{tikzcd}
    	{\Cl[\mc D]} & {\Cl[\mc A]} \\
    	{\Cl[\mc B]} & {\Cl[\mc D]}
    	\arrow["{{\Cl[u]}}", from=1-1, to=1-2]
    	\arrow["{{\Cl[v]}}"', from=1-1, to=2-1]
    	\arrow[""{name=0, anchor=center, inner sep=0}, "{{\Cl[f]}}", from=1-2, to=2-2]
    	\arrow[""{name=1, anchor=center, inner sep=0}, "{{\Cl[g]}}"', from=2-1, to=2-2]
    	\arrow[curve={height=-6pt}, between={0.2}{0.8}, Rightarrow, from=1, to=0]
    \end{tikzcd}\]
    By~\cref{fconsCL}, $\Cl[f]$ will be dominant. It is well-known that the geometric morphisms between coherent topoi induced by maps of pretopoi are spartan (cf.~\cite{SGA4}). Thus by~\cref{sparBC}, $\Cl[v]$ will also be dominant, thus $v$ will be f-conservative by~\cref{fconsCL} again. 
\end{proof}

\begin{rem}
    In fact, one can also directly show that t-conservative maps are closed under cocomma in pretopoi by using the result about lax descent for essential geometric morphisms and \emph{op-dominant} geometric morphisms, i.e.\ those whose inverse images are t-conservative. However, in this case by considering the classifying frame construction is not enough, since in general $\Cl[f]$ for a map $f$ in $\DL$ will not be essential. Instead one can use the locale of filters construction described in~\cite{pitts1983amalgamation}. This strategy is used e.g. in~\cite{pitts2020interpolation} to show the interpolation result for pretopoi. 
\end{rem}

\bibliography{thebib}
\bibliographystyle{alpha}

\end{document}